\documentclass[12pt]{article}
\usepackage{amsmath,amsthm,amscd,amsfonts,amssymb, mathrsfs}
\usepackage{latexsym}
\usepackage[usenames]{color}
\newcommand{\CC}{\mathbb C}

\newcommand{\NN}{\mathbb N}

\newcommand{\RR}{\mathbb R}

\newcommand{\ZZ}{\mathbb Z}
\newcommand{\EE}{\mathbb E}

\newcommand{\dd}{\mathcal D}

\newtheorem{thm}{Theorem}[section]
\newtheorem{lemma}[thm]{Lemma}

\def\half{\frac{1}{2}}

\def\pf{{\noindent \bf Proof: }}
\DeclareMathOperator*{\supess}{ess\,sup}
\def\etan{{\mathcal N}}

\usepackage[T1]{fontenc}
\usepackage{graphicx,graphics}
\usepackage{relsize}
\usepackage[all]{xy}
\usepackage{tikz}
\usepackage{color}
\allowdisplaybreaks

\newcommand{\GG}{\mathbb{G}}
\newcommand{\Hi}{\mathscr{H}}
\newcommand{\Ex}{\mathop{\mathbb{E}}}

\newcommand{\norm}[1]{\left\lVert #1\right\rVert}

\newtheorem{theorem}{Theorem}[section]

\newtheorem{remark}[theorem]{Remark}
\numberwithin{equation}{section}

\newtheorem{hypothesis}[theorem]{Hypothesis}

\date{\small \today}

\begin{document}
\setlength{\belowdisplayskip}{2pt} \setlength{\belowdisplayshortskip}{2pt}
\setlength{\abovedisplayskip}{2pt} \setlength{\abovedisplayshortskip}{2pt}

\vspace{-2cm}
\begin{center}
{\large Regularity of the density of states of Random Schr\"odinger Operators}

\vspace{.5cm}

\small Dhriti Ranjan Dolai\\
\small Department of Mathematics \\
\small Indian Institute of Technology, Dharwad \\
\small Dharwad-.580011, India\\
\small email:  dhriti@iitdh.ac.in \\
\vspace{.5cm}
\small M Krishna \\
\small Ashoka University, Plot No 2, Rajiv Gandhi Education City\\
\small Rai, Haryana 131029, India \\
\small email: krishna.maddaly@ashoka.edu.in\\
\vspace{.5cm}
\small Anish Mallick\\
\small Department of Mathematics \\
\small Pontificia Universidad Cat\'{o}lica de Chile. \\
\small Vicuna Mackenna 4860, Macul , Santiago, Chile\\
\small email: anish.mallick@mat.uc.cl


\abstract{\small In this paper we solve a long standing open problem
for Random Schr\"odinger operators on $L^2(\RR^d)$ with i.i.d single site
random potentials. We allow a large class of free operators, including
magnetic potential, however our method of proof works only for the
case when the random potentials satisfy a complete covering condition.
We require that the supports of the random potentials cover $\RR^d$
and the bump functions that appear in the random potentials form
a partition of unity.
For such models, we show that the Density of States (DOS)
 is $m$ times differentiable in the part of the spectrum where exponential
localization is valid, if the single site distribution
has compact support and has H\"older continuous $m+1$ st derivative.
The required H\"older continuity
depends on the fractional moment bounds satisfied by appropriate operator
kernels. Our proof of the Random Schr\"odinger operator case 
is an extensions of our proof 
for  Anderson type models
on $\ell^2(\GG)$, $\GG$ a countable set, with the property that
the cardinality of the set of points at distance $N$ from any fixed
point grows at some rate in $N^\alpha, \alpha >0$. This condition rules out
the Bethe lattice, where our method of proof works but the degree
of smoothness also depends on the localization length,  
a result we do not present 
here. Even for these models the random
potentials need to satisfy a complete covering condition. The 
Anderson model on the
lattice for which regularity results were known earlier
also satisfies the complete covering condition. }

\end{center}

\section{Introduction}
In the study of the Anderson Model and  Random Schr\"odinger
operators,  modulus of continuity of the
Integrated Density of States (IDS) is well understood, (see Kirsch-Metzger \cite{MR2307751} for a comprehensive review).
In dimension bigger than one,  there are very few results  on further 
smoothness of the IDS, even
when the single site distribution is assumed to have more smoothness, except
for the case of the Anderson model itself at high disorder, (see
for example Campanino-Klein \cite{MR836001}, Bovier-Campanino-Klein-Perez \cite{MR929139}, Klein-Speis \cite{MR1033921}, Simon-Taylor \cite{MR814540}).

In this paper we will show, in Theorems \ref{thm:smoothnessThmDiscreteCase} and \ref{thm:smoothnessThmContCase},  that the IDS is almost
as smooth as the single site distribution
for a large class of continuous and discrete random operators.  These are
\begin{equation}\label{cont} 
H^\omega = H_0 + \sum_{n \in \ZZ^d} \omega_n u_n,
\end{equation}
on $L^2(\RR^d)$ and
\begin{equation}\label{disc} 
h^\omega = h_0 + \sum_{n \in \GG} \omega_n P_n,
\end{equation}
on the  separable Hilbert space $\Hi$ and a countable set $\GG$.
The operator $h_0$ is a bounded
self-adjoint operator and the $\{P_n\}$  are finite rank projection.
We specify the conditions
on $H_0, h_0, u_n, P_n$ and $\omega_n$ in the following sections.

The IDS, denoted $\etan(E)$, is the distribution function of a non-random
measure obtained as the
weak limit of a sequence of random atomic
measures.  The proof of the existence of such limits for various models
of random operators has  a long history.
These results are well documented in the books of
Carmona-Lacroix \cite{MR1102675}, Figotin-Pastur \cite{MR1223779}, Cycon et.al. \cite{MR883643}, Kirsch \cite{MR2509110}
and the reviews of  Kirsch-Metzger
\cite{MR2307751}, Veseli\'{c} \cite{MR2378428} and in a review
for stochastic Jacobi matrices by Simon \cite{MR901051}.
In terms of the projection
valued spectral measures $E_{H^\omega}, E_{h^\omega}$ associated with
the self-adjoint operators $H^\omega, h^\omega$, the function
$\etan(E)$ has an explicit expression, for the cases when $h^\omega$, $H^\omega$ are ergodic.  
For the model (\ref{cont})  it is given as
$$
\frac{1}{\int u_0(x) dx} \EE \Bigl[ tr\Bigl(u_0 E_{H^\omega}((-\infty, E])\Bigr)\Bigr]
$$
and for the model (\ref{disc}) it turns out to be 
$$
 \frac{1}{tr(P_0)} \EE \bigg[ tr\bigg(P_0 E_{h^\omega}  ((-\infty, E])\bigg) \bigg].
$$

We note that by using the same symbol $\etan$ for two different models,
we are abusing notation but this abuse will not cause any 
confusion as the contexts are
clearly separated to different sections. 
The first of these expressions for the IDS 
is often called the {\it Pastur-Shubin trace formula}.

In the case of the model (\ref{cont}) in dimensions $d \geq 2$,
there are no results in the literature on the smoothness
of $\etan(E)$, our results are the first to show even continuity of the
density of states (DOS), which is the derivative of $\etan$ almost 
every $E$.
The results of 
Bovier et.al. in \cite{MR929139}
are quite strong for the Anderson model at large disorder
and it is not clear that their proof
using supersymmetry extends to other discrete random operators.

In the one dimensional Anderson model, Simon-Taylor \cite{MR814540} showed that
$\etan(E)$ is $C^\infty$ when the single site distribution (SSD) is
compactly supported and is H\"older continuous.
Subsequently,  Campanino-Klein \cite{MR836001} proved
that $\etan(E)$ has the same degree of smoothness as the SSD.
In the one dimensional strip, smoothness results were shown by
 Speis \cite{MR1186042, MR2636152}, Klein-Speis \cite{MR1033921, MR952881},
Klein-LaCroix-Speis \cite{MR1031404},  Glaffig \cite{MR1069256}.
For some non-stationary random potentials on the lattice, Krishna
\cite{MR1894550} proved smoothness for an averaged total spectral measure.

There are several results showing $\etan(E)$ is analytic
for the Anderson model on $\ell^2(\ZZ^d)$. 
Constantinescu-Fr\"{o}hlich-Spencer \cite{MR748803}
showed analyticity of $\etan(E)$ when SSD is analytic.
The result of Carmona \cite[Corollary VI.3.2]{MR1102675}) improved
the condition on SSD to requiring fast exponential decay to get analyticity.
In the case of the Anderson model over $\ell^2(\ZZ^d)$  at large disorder the
results of Bovier et.al. \cite{MR929139} give smoothness
of $\etan(E)$ when the Fourier transform $h(t)$ of the SSD is $C^\infty$ and
the derivatives decay like $1/t^\alpha$ for some $\alpha > 1$ at infinity.
They also give variants of these, in particular if the SSD is $C^{n+d}$ then
$\etan(E)$ is $C^n$ under mild conditions on its decay at $\infty$. They also
obtain some analyticity results.
Acosta-Klein \cite{MR1184777} show that $\etan(E)$ is analytic on the Bethe lattice for SSD close to the Cauchy distribution.
While all these results are valid in the entire spectrum,
Kaminaga et. al. \cite{MR2992798} showed local analyticity
of $\etan(E)$  when the SSD has an analytic component in an interval allowing
for singular parts elsewhere, in particular for the uniform distribution.
Analyticity results obtained by March-Sznitman \cite{MR879550} were similar
to those of Campanino-Klein \cite{MR836001}.

In all the above models, only when $E$ varies in the pure point spectrum 
that regularity of $\etan(E)$ beyond Lipshitz continuity is shown.  
This condition that $E$ has to be in the pure point spectrum may not
have been explicitly stated, but it turns out to be a consequence
of the assumptions on disorder or assumptions on the dimension 
in which the models were
considered.  For the Cauchy distribution in the Anderson model on
$\ell^2(\ZZ^d)$, Carmona-Lacroix
\cite{MR1102675} have a theorem showing analyticity in the entire spectrum.
However, absence of pure point spectrum is only a conjecture in these
models as of now. At the time of revision of this paper one of us
Krishna \cite{krishna} could show that in the Anderson model on the 
Bethe lattice analyticity of the density of states with Cauchy distribution
is valid at all disorders as part of a more general result. This result
in particular exhibits regularity of the density of states through the
mobility edge in the Bethe lattice case.

In the case of random band matrices, with
the random variables following a Gaussian distribution, Disertori-Lager \cite{MR3665217},
Disertori \cite{MR2114358, MR2013666}, Disertori et. al. \cite{MR1942858} have smoothness results for an appropriately defined density
of states.   Recently Chulaevsky \cite{Chu} proved infinite
smoothness for non-local random interactions.

For the one dimensional ergodic random operators IDS was shown to be
 log H\"older continuous by Craig-Simon \cite{MR714434}.  There
are numerous results giving the modulus of continuity of $\etan(E)$,
one of which is the Wegner \cite{wegner1981bounds} estimate, for independent random potential,
showing its Lipschitz continuity.
Combes et.al. in \cite{MR2362242} showed that for Random Scr\"odinger
operators with independent random potentials, the modulus of continuity
of $\etan(E)$ is the same as that of the SSD.  
For non i.i.d potentials in higher dimensions there are some results
on modulus of continuity for example that of Schlag \cite{MR1860759} showing
and by Bourgain-Klein \cite{MR3103255} who show log H\"older continuity for
the distribution functions of outer measures for a large class
of random and non-random Schr\"odinger operators. 
We refer to these papers for more recent results on the continuity
of $\etan(E)$ not given in the books cited earlier.

The idea of proof of our Theorems is the following.  Suppose
we have a self-adjoint matrix $A^\omega$ of size $N$ with i.i.d real valued
random variables $\{\omega_1, \dots, \omega_N\}$ on the diagonal
with each $\omega_j$ following the distribution $\rho(x)dx$.
Then the average of the matrix elements of the resolvent of $A^\omega$
are given by
$$
f(z) = \int (A^\omega - zI)^{-1}(i, i) ~ \prod_{k=1}^N \rho(\omega_k)d\omega_k,
$$
for any $z \in \CC^+$.
We take $z = E+i\epsilon, ~ \epsilon >0$, then we see that 
from the definitions, the function $(A^\omega - zI)^{-1}(i, i)$
can be written as a function of $\vec{\omega} - E\vec{1}$ and $\epsilon$,
namely,
\begin{eqnarray*}
&F(\vec{\omega} - E \vec{1}, \epsilon) = (A^\omega - zI)^{-1}(i, i), ~
\Phi(\vec{\omega}) = \prod_{i=1}^N\rho(\omega_i) \\
&\vec{\omega} = (\omega_1, \omega_2, \dots, \omega_N), ~ \vec{1} = (1, 1, \dots, 1).
\end{eqnarray*}
Then it is clear
that with $*$ denoting convolution of functions on $\RR^N$ and
 setting $\tilde{F_\epsilon}(x) = F(-x, \epsilon)$,
$$
\EE\left((A^\omega - zI)^{-1}(i, i) \right) = (\tilde{F_\epsilon} * \Phi)(E \vec{1}).
$$
Since convolutions are smoothing, we get the required smoothness as a
function of $E$ if one of the components $\tilde{F_\epsilon}$ or $\Phi$
is smooth on $\RR^N$.  Since we are assuming that each $\rho$ has a degree
of smoothness, which passes on to $\Phi$, we get a smoothness result for 
operators with finitely many random variables having the above form. 

Let us remark here that it is in this step, which is crucial for 
further analysis, 
that we need a {\it complete covering condition}, even for
finite dimensional compressions of our random operators be they continuous
or discrete.

If we were to replace $A^\omega$ by an operator with infinitely
many random variables $\omega_i$,
we would encounter the problem of concluding smoothing properties
of "convolutions of" functions of infinitely many variables. 
This is an important difficulty that needs to be solved.

One of the interesting aspects of the operator $H^\omega$ (or $h^\omega$)
 we are dealing with is that
there is a sequence of operators (denoted by $A^\omega_k$),
containing finitely many random variables $\omega_i$,  which
converges to $H^\omega$ (or $h^\omega$)  in strong resolvent sense.
Hence we can write the limit as a telescoping sum, namely,
$$
(A^\omega - z)^{-1}(i,i)  = (A_1^\omega - z)^{-1}(i,i) + \sum_{k=1}^\infty \Bigl[ (A_{k+1}^\omega - z)^{-1}(i,i) - (A_k^\omega - z)^{-1}(i,i) \Bigr].
$$
Since the operators appearing in the summands all contain finitely many 
$\omega_i$ their averages over the random variables can be written 
as convolutions of functions of finitely many variables $\omega_i$. 
Then, most of the work in the proof is to show that quantities of the form
$$\left|\int  \bigl[ (A_{k+1}^\omega - z)^{-1}(i,i) - (A_k^\omega - z)^{-1}(i,i) \bigr] \left(\sum_{j=1}^{N_{k+1}}\frac{\partial}{\partial \omega_j}\right)^l \prod_{n=1}^{N_{k+1}} \rho(\omega_n)d\omega_n\right|$$
with $N_k$ growing at most as a fixed polynomial in $k$, are summable in $k$.
This is the part where we use the fact that we are working in
the localized regime, where it is possible to show that
they are exponentially small in $k$. 

For the discrete case the procedure is relatively straight forward
and there are no
major technical difficulties to overcome, but in the continuous case, the
infinite rank perturbations pose a problem, since the trace of the Borel-Stieltjes
transform of the average spectral measures do not converge.
We overcome this problem by renormalizing this transform appropriately.
For our estimates to work, we have to use
fractional moment bounds and also uniform bounds on the integrals of
resolvents.  Both of these are achieved because we have dissipative operators
(up to a constant) whose resolvents can be written in terms of integrals of contraction
semigroups.

As stated above, our proof is in the localized regime.
The Anderson model 
was formulated by Anderson \cite{and} who argued that there is no
diffusion in these models for high disorder or at low energies.  
The corresponding spectral statement is that there is only pure point
spectrum or only localization for these cases.
In the one dimensional
systems, where the results are disorder independent,
 localization was shown rigorously by Goldsheid et.al. \cite{MR0470515}
for random Schr\"odinger operators and by Kunz-Souillard \cite{MR597748}
for the Anderson model.  For higher dimensional Anderson model the
localization was proved simultaneously by
Fr\"{o}hlich et.al. \cite{MR814541},
Simon-Wolff \cite{MR820340}, Delyon et.al. \cite{MR806247}
based on exponential decay shown by Fr\"{o}hlich-Spencer \cite{MR839666}
who introduced a tool called multi scale analysis in the discrete case.
 A simpler proof
based on exponential decay of fractional moments was later given by
Aizenman-Molchanov \cite{MR1244867}.  There are numerous improvements
and extensions of localization results beyond these papers.

In the case of continuous models, Combes-Hislop \cite{MR2362242,CH3}, Klopp \cite{K3}, Germinet-Klein \cite{GK4},
Combes-Hislop-Tip \cite{CHT1}, Bourgain-Kenig \cite{B1} and Germinet-De Bievre \cite{germinet1998dynamical}
provided proof of localization for different types of models.
The fractional moment method was first extended to the continuous case 
in Aizenman et.al. in \cite{MR2207021}
and later improved by Boutet de Monvel et.al. \cite{de2006localization}.

We refer to Stollmann \cite{MR1935594} for the numerous advances that
followed on localization. 

The rest of the article is divided into three parts. Section 2 has  all
the preliminary results, which will be used significantly for both
the discrete and the continuous case.
Section 3 will deal with the discrete case, where we use a method of proof
which will be reused for the continuous case.
The main result of Section 3 is Theorem \ref{thm:smoothnessThmDiscreteCase}
 which in the case of Anderson tight binding model would prove the regularity
of density of states.
Finally in Section 4 we will deal with the random Schr\"odinger operators
 and the main result there is Theorem \ref{thm:smoothnessThmContCase}.

\section{Some Preliminary Results}

In this section we present some general results that are at the heart
of the proofs of our theorems. 
These are Theorem \ref{thm:finitesmooth} and
 Theorem \ref{thm:remarkable}.  The later theorem, stated for functions, gives
a bound of the form 
$$\left| \int \left(\frac{1}{x-w} - \frac{1}{x -z}\right)f(x) dx\right| \leq C_{f,s} |z-w|^s$$
for certain family of $f$.
For operators, we need more  work and need more uniformity for $f$.

The first theorem is quite general and is about random perturbations of
self-adjoint operators and their smoothing properties of complex
valued functions of the operators.  

\begin{thm}\label{thm:finitesmooth}
Consider a self-adjoint operator $A$ on a separable Hilbert space $\Hi$ 
and let  $\{T_n\}_{n=1}^N, N < \infty$ be bounded positive operators
such that $\sum_{n=1}^N T_n = I$, where $I$ denotes the identity operator on $\Hi$.
Suppose $\{\omega_n, n=1, \dots, N\}$ are independent real valued random variables distributed according to
$\rho_n(x) dx$ and consider the random operators $A^\omega = A + \sum_{n=1}^N \omega_n T_n$.  If $f$ is a complex valued function on the set of linear operators
on $\Hi$, such that $f(A^\omega - E I)$ is a bounded measurable function
of $(\omega_1, \dots, \omega_n, E)$, then 
$
h(E) = \EE\big[ f(A^\omega - EI)\big] $ satisfies 
$h \in C^m(\RR) , ~~ \mathrm{if} ~~ \rho_n \in C^m(\RR)$ and $\rho_n^{(k)} \in L^1(\RR), ~n=1,2,\cdots,N$ and $ 0\leq k\leq m$.
\end{thm}
\pf Using the conditions on $\{T_n\}$ we see that $A^\omega - EI = 
A + \sum_{n=1}^N (\omega_n - E) T_n$.  Thus 
$ f(A^\omega - E I)$ is a bounded measurable function of the variables
$(\omega_1 - E, \omega_2 - E, \dots, \omega_n -E)$, which is a point
$\vec{\omega} - E \vec{1}$ in $\RR^N$, where $\vec{1} = (1, \dots, 1)$,
we write $F(\vec{\omega}- E\vec{1}) = f(A^\omega - EI)$. 
Then the expectation can be written as 
$$
\EE [f(A^\omega - EI)] = \int_{\RR^N} F(\vec{\omega} - E \vec{1}) \Phi(\vec{\omega})d\vec{\omega} 
 = \int_{\RR^N} F(-( E \vec{1} -\vec{\omega})) \Phi(\vec{\omega})d\vec{\omega}, 
$$
where we set $\Phi(\vec{\omega}) = \prod_{n=1}^N \rho_n(\omega_n)$.  Writing
now $g(\vec{x}) = F(-\vec{x})$ we see that 
$$
\EE [f(A^\omega - EI)] = (g*\Phi)(E\vec{1}),
$$
where $*$ denotes convolution in $\RR^N$. The result now follows easily 
from the properties of convolution of functions on $\RR^N$.

\qed

For later use we note that if $\nabla$ denotes the gradient operator
on differentiable functions on $\RR^N$ and ${\bf D}$ denotes
${\bf D}\Phi = \nabla \Phi \cdot \vec{1} = \sum_{j=1}^N \frac{\partial}{\partial x_i}\Phi$, then an integration by parts yields
\begin{equation}\label{nabalform}
\frac{d^\ell }{dE^\ell}h(E)=\frac{d^\ell}{dE^\ell}( g*\Phi)(E \vec{1}) =  (g * ({\bf D}^\ell \Phi))(E\vec{1}).
\end{equation}

\begin{remark}\label{completecovering} 
This theorem clarifies why the complete covering condition
is needed in main our results for the discrete and the continuous models.
The covering property is needed even
for obtaining smoothness of finitely many random perturbations of a
self-adjoint operator, while such a property is not needed for modulus
of continuity results.  We are unsure at the moment if this condition can be relaxed.
\end{remark}

Let $A, B $ be self-adjoint operators and let $F_1, F_2$ be 
bounded non-negative operators
on a separable Hilbert space $\Hi$.  For $X \in \{A, B\}, ~ z \in \CC^+$,
set,
$$
R(X, x, y, z) = (X+xF_1 + yF_2 -z)^{-1} 
$$
and 
$$
R(X, x, z) = (X+xF -z)^{-1},~~ F = F_1 + F_2
$$
for the following Theorem.  
For the rest of the paper by a {\it smooth
indicator function } on an interval $(a, b)$ we mean a smooth function which is 
one in $[c, d] \subset (a, b)$ which vanishes on $\RR \setminus (a, b)$ with
$a - c + b-d$ as small as one wishes.  

\begin{theorem}\label{thm:remarkable}
Suppose $A, B, F_1, F_2, F, z$ and $\Hi$ be as above.  Suppose
$\rho_1, \rho_2$ are  compactly supported functions 
on $\RR^+$ such that their derivatives are $\tau$-H\"older continuous
and their supports are contained in $(0, R)$.  Let $\chi_R$ denote a smooth 
indicator function of the set $(0, 2R+1)$ and let 
$\phi_R(x) = \chi_R(x+ \frac{5}{2}R +1)$.
Then for any $0 < s < \tau$ and some  constant $\Xi$  
(depending upon $\rho_1, \rho_2, s, \tau$ but independent of $z, A, F_1, F_2$),
\begin{enumerate}
\item
\begin{eqnarray}\label{eq:remarkable}
&\displaystyle{\norm{\int F^\half \bigg( R(A, x_1, x_2, z) - R(B, x_1, x_2, z) \bigg)F^\half \rho_1(x_1) \rho_2(x_2) dx_1 dx_2 }} \nonumber \\
&\leq \displaystyle \Xi {\int \norm{ F^\half \bigg(R(A, x_1, x_2, z) - R(B, x_1, x_2, z) \bigg)F^\half}^s}, \nonumber \\
& ~~~~~ \times \phi_R(x_1 )\phi_R(x_2) dx_1 dx_2 . 
\end{eqnarray}
\item  Specializing to the case when $F_1 = F_2, x_1 = x_2 = x/2$ we have
\begin{align}\label{eq:remarkable2}
&\displaystyle{\norm{\int F^\half \bigg( R(A, x, z) - R(B, x, z) \bigg)F^\half \rho_1(x) dx }} \\
&\leq \displaystyle  \Xi {\int \norm{ F^\half \bigg(R(A, x, z) - R(B, x, z) \bigg)F^\half}^s \phi_R\left(x \right) dx  }. \nonumber
\end{align}
\end{enumerate}

\end{theorem}

\begin{remark}
\label{dir-int}
 The integrals appearing in \eqref{eq:remarkable} and \eqref{eq:remarkable2} are viewed as operators in the sense of direct integrals (see \cite[Theorem XIII.85]{RMBS}).
This is the case because $X+ x_1 F_1+x_2 F_2$ is decomposable on
$$L^2\big(\mathbb{R}^2,\prod_i\rho(x_i)dx_i, \mathcal{H}\big).$$
Hence all the integrals of this operator valued function, that appear in the proof, are well-defined in the sense of direct integral representation.

\end{remark}

\begin{proof}
We define
$$A^{t}=A+t~F, ~~ B^{t}=B+ t~F,\qquad\forall -2R-1<t<-2R.$$ 
Then, we have the equality,
\begin{equation}\label{eq201}
A + x_1 F_1 + x_2 F_2=A^t +\left(\frac{x_1-x_2}{2}\right) (F_1 - F_2)+\left(\frac{x_1 + x_2}{2}-t\right)F .\end{equation}
Using the resolvent equation, we have, with $F_- = F_1 - F_2$,
\begin{align}\label{eq202}
& R(A, x_1, x_2, z) =\left(A^{t}+\left(\frac{x_1-x_2}{2}\right) F_- -z\right)^{-1}\nonumber \\
&-\left(\frac{x_1+x_2}{2}-t\right) R(A, x_1, x_2, z)F \left(A^{t}+\left(\frac{ x_1 -x_2}{2}\right) F_- -z\right)^{-1}
\end{align}
which can be re-written (using the notation $\tilde{A}^{t}=A^{t}+\left(\frac{x_1 - x_2}{2}\right) F_- $) as
\begin{align}\label{eq203}
&\sqrt{F}R(A, x_1, x_2, z)\sqrt{F}=\frac{1}{\frac{x_1+x_2}{2}-t}I\nonumber \\
&\qquad-\frac{1}{\left(\frac{x_1+x_2}{2}-t\right)^2}\left(\frac{1}{\frac{x_1+x_2}{2}-t} I+\sqrt{F} \left(\tilde{A}^{t}-z\right)^{-1} \sqrt{F}\right)^{-1}.
\end{align}
($I$ is the identity operator on the range of $\sqrt{F}$)
Similar relations hold for $B$, where $B^t, \tilde{B}^t$ are defined by
replacing $A$ with $B$ in the equations (\ref{eq201} - \ref{eq203}).
We set
$$
\tilde{R}^t_{A, z} = \sqrt{F}(\tilde{A}^t -z)^{-1}\sqrt{F}, ~~ \tilde{R}^t_{B, z} = \sqrt{F}(\tilde{B}^t -z)^{-1}\sqrt{F}. 
$$
Then using equation \eqref{eq203} we get the relation,
\begin{align}\label{eq:New300}
&\int \sqrt{F}(R(A, x_1, x_2, z)-R(B, x_1, x_2, z))\sqrt{F} ~\rho_1(x_1)\rho_2(x_2) dx_1 dx_2 \nonumber \\
&=\int \left[  \left(\frac{1}{\frac{x_1+x_2}{2}-t} I+ \tilde{R}^{t}_{A,z}\right)^{-1}-\left( \frac{1}{\frac{x_1+x_2}{2}-t} I+\tilde{R}^{t}_{B,z}\right)^{-1} \right] \nonumber  \\ 
&\qquad \frac{1}{\left(\frac{x_2+x_2}{2}-t\right)^2}\rho_1(x_1)\rho_2(x_2) dx_1 dx_2 \nonumber \\
&=2\int \left[  \left(\gamma I+ \tilde{R}^{t}_{A,z}\right)^{-1}-\left( \gamma I+\tilde{R}^{t}_{B,z}\right)^{-1} \right] \nonumber  \\ 
&\qquad\qquad\qquad \rho_1\left(t+\frac{1}{\gamma}+\eta\right)\rho_2\left(t+\frac{1}{\gamma}-\eta\right)d\gamma d\eta
\end{align}
where $\gamma=\left(\frac{x_1+x_2}{2}-t\right)^{-1}$ and $\eta=\frac{x_1-x_2}{2}$. 
For $X$ self-adjoint,  $\tilde{R}^{t}_{X,z}$ is an  operator valued Herglotz function 
and its imaginary part is a positive operator for $\Im(z) >0$. 
Hence the operators  $ \left(\gamma I+ \tilde{R}^{t}_{X,z}\right)$ 
generates a strongly continuous one parameter semi-group, and we can apply the 
Lemma \ref{lem:interchangingIntegrals} for the $\gamma$ integral, and then the 
do the $\eta$ integral to get
\begin{align}\label{eq:New301}
& \int \left[  \left(\gamma I+ \tilde{R}^{t}_{A,z}\right)^{-1}-\left( \gamma I+\tilde{R}^{t}_{B,z}\right)^{-1} \right]\nonumber \\ 
&\qquad\qquad\qquad\qquad\qquad \rho_1\left(t+\frac{1}{\gamma}+\eta\right)\rho_2\left(t+\frac{1}{\gamma}-\eta\right)d\gamma d\eta\nonumber\\
&=-\int \left[ \int_0^\infty \left(e^{i w\left(\gamma I+ \tilde{R}^{t}_{A,z}\right)}-e^{i w\left(\gamma I+ \tilde{R}^{t}_{B,z}\right)}\right)dw \right]\nonumber\\
&\qquad\qquad\qquad\qquad\qquad \rho_1\left(t+\frac{1}{\gamma}+\eta\right)\rho_2\left(t+\frac{1}{\gamma}-\eta\right)d\gamma d\eta\nonumber\\
&=-\int \int_0^\infty \left(e^{i w \tilde{R}^{t}_{A,z}}-e^{i w \tilde{R}^{t}_{B,z}}\right)~  e^{i \gamma w }\rho_1\left(t+\frac{1}{\gamma}+\eta\right)\rho_2\left(t+\frac{1}{\gamma}-\eta\right)d\gamma ~dw d\eta,
\end{align}
which can be bounded as 
\begin{align}\label{eq:lemBddonGreenFuncEq120}
& \bigg\|\int \int_0^\infty \left[ e^{i w \tilde{R}^{t}_{A,z}}-e^{i w \tilde{R}^{t}_{B,z}}\right]\nonumber\\
&\qquad\qquad\qquad e^{i \gamma w }\rho_1\left(t+\frac{1}{\gamma}+\eta\right)\rho_2\left(t+\frac{1}{\gamma}-\eta\right)d\gamma ~dw d\eta \bigg\|\nonumber \\
&\qquad\leq \int \norm{\left(e^{i w \tilde{R}^{t}_{A,z}}-e^{i w \tilde{R}^{t}_{B,z}}\right) }\nonumber\\
&\qquad \qquad\qquad \left|\int e^{ i\gamma w }\rho_1\left(t+\frac{1}{\gamma}+\eta\right)\rho_2\left(t+\frac{1}{\gamma}-\eta\right)d\gamma\right| dw d\eta.
\end{align}
 The assumption we made on the supports of $\rho_1, \rho_2$ implies that
 $-\frac{R}{2}<\eta<\frac{R}{2}$, and the choice ${-2R-1<t<-2R}$ implies $-\frac{5}{2}R-1<t\pm \eta<-\frac{3R}{2}$. 
This implies that 
$$
\bigg\{\gamma:\psi_{t,\eta}(\gamma)\neq 0, -\frac{5}{2}R-1<t\pm \eta<-\frac{3R}{2}\bigg\}  \subset \bigg(\frac{2}{2+7R},\frac{2}{3R}\bigg),
$$
where $\psi_{t,\eta}(\gamma) = \rho_1\left(t+\frac{1}{\gamma}+\eta\right)\rho_2\left(t+\frac{1}{\gamma}-\eta\right)$.
Thus for fixed $t, \eta$, the function $\psi_{t,\eta}(\gamma)$  is of compact support and has a 
$\tau$-H\"older continuous derivative as a function of $\gamma$, for the  $\tau$ stated as in the Theorem. 
Also, the derivative of $\psi_{t,\eta}$ is uniformly $\tau$-H\"older continuous 
and the constant is bounded uniformly in $t,\eta$, which follows from the support properties of $\psi_{t,\eta}$ and the bounds on $t,\eta$. 
Therefore, if we denote the Fourier transform of $\psi_{t,\eta}(-\gamma)$ by 
$\widehat{\psi_{t,\eta}}$, then standard Fourier analysis gives the bound, 
\begin{align*}
& \left|\int e^{i \gamma w }\rho_1\left(t+\frac{1}{\gamma}+\eta\right)\rho_2\left(t+\frac{1}{\gamma}-\eta\right)d\gamma\right|\\
&\qquad \leq \frac{C}{|w|^{1+\tau}}\left( \norm{|w|^{1+\tau} \widehat{\psi_{t,\eta}}(w)}_\infty  \right)  \leq \frac{\tilde{C}}{|w|^{1+\tau}} ~~ for~|w|\gg 1
\end{align*}
for some $\tilde{C}$ independent of $t, \eta$ but depends on $\rho_1, \rho_2$
.
Again using the bounds on $t,\eta$ and $\gamma$, we see that for 
small $|w|$, the $w$ integral is bounded uniformly in $t,\eta$,
by the $L^\infty$ norm of $\rho_1$ and $\rho_2$ and hence
$\tilde{C}$ is independent $t, \eta$ for all $w$. 

On other hand using  the Lemma \ref{lem:semiGrpNrmbdd}, we have 
\begin{align*}
\norm{ e^{ iw \tilde{R}^{t}_{A,z}}-e^{i w \tilde{R}^{t}_{B,z}} }\leq 2^{1-s}|w|^s\norm{\tilde{R}^{t}_{A,z}-\tilde{R}^{t}_{B,z}}^s
\end{align*}
for $0<s<1$. By choosing $s< \tau /2 $ and using above bounds in 
\eqref{eq:lemBddonGreenFuncEq120} we have
\begin{align}\label{eq:New302}
& \bigg\|\int \int_0^\infty \left(e^{ i w \tilde{R}^{t}_{A,z}}-e^{i w \tilde{R}^{t}_{B,z}}\right)\nonumber\\
&\qquad\qquad\qquad e^{ i\gamma w }\rho_1\left(t+\frac{1}{\gamma}+\eta\right)\rho_2\left(t+\frac{1}{\gamma}-\eta\right)d\gamma ~dw d\eta \bigg\|\\
&\qquad\leq \hat{C} \left(1+\int_1^\infty \frac{1}{w^{1+\tau-s}}dw \right) \int  \norm{\tilde{R}^{t}_{A,z}-\tilde{R}^{t}_{B,z}}^{s} d\eta,
\end{align}
The integral we started with is independent of $t$ so we can integrate it
with respect to the Lebesgue measure on an interval of length one.
Therefore, combining  the inequalities (\ref{eq:New300}, \ref{eq:New301}, \ref{eq:lemBddonGreenFuncEq120}, \ref{eq:New302}) and integrating $t$ over an interval of length 1, yields 

\begin{align*}
&\norm{\int \sqrt{F}(R(A, x_1, x_2,z) -R(B, x_1, x_2, z) \sqrt{F} ~\rho_1(x_1)\rho_2(x_2) dx_1 dx_2 }\\
&\qquad= \int_{-2R-1}^{-2R} \norm{\int \sqrt{F}(R(A, x_1, x_2,z) -R(B, x_1, x_2, z) \sqrt{F} ~\rho_1(x_1)\rho_2(x_2) dx_1 dx_2 } dt \\
&\qquad\leq C  \int^{-2R}_{-2R-1} \int_{-\frac{R}{2}}^{\frac{R}{2}}   \norm{\tilde{R}^{t}_{A,z}-\tilde{R}^{t}_{B,z}}^{s} d\eta dt \\
&\qquad \leq C \int\int \bigg\| \sqrt{F}\left(A+\hat{x}_1 F_1+\hat{x}_2 F_2 -z\right)^{-1}\sqrt{F}\\
&\qquad\qquad\qquad\qquad - \sqrt{F}\left(B +\hat{x}_1 F_1 +\hat{x}_2 F_2-z\right)^{-1}\sqrt{F}\bigg\|^s \phi_R(\hat{x}_1)\phi_R(\hat{x}_2) ~  d\hat{x}_1 d\hat{x}_2.
\end{align*}
For the last inequality we used the definition of $\tilde{R}^{t}_{X,z}$ 
changed variables $\hat{x}_1=t+\eta, ~ \hat{x}_2=t-\eta$ along
with a slight increase in the range of integration to accommodate 
the bump $\phi_R$ to have their supports in $(-\frac{5}{2}R -1, -\frac{R}{2})$.
\end{proof}

\section{The Discrete case}
 Let $\GG$ denote a un-directed connected graph with a graph-metric $d$. Let $\{x_n\}_n$ denote an enumeration of $\GG$ satisfying
$d(\Lambda_N,x_{N+1})=1$ for any $N \in\NN$, where 
\begin{equation}\label{lambdaM}
\Lambda_N=\{x_n: n\leq N\}, ~~ \Lambda_\infty = \GG, 
\end{equation}
and
\begin{equation}\label{growthofG}
\liminf_{N\rightarrow\infty} \frac{d(x_0,\GG\setminus\Lambda_N)}{g(N)} = r_\GG >0,
\end{equation}
for some increasing function $g$ on $\RR^+$.
Typically, we will have $g(N) = N^{1/d}$ for $\GG = \ZZ^d$ and $g(N) = \log_{K}(N) $
for the Bethe lattice with connectivity $K > 2$.
Henceforth for indexing $\GG$ we will say $n \in \GG$ to mean $x_n \in \GG$.

Let $\Hi$ be a complex separable Hilbert space equipped with a countable
family $\{P_n\}_{n\in \GG}$ of finite rank orthogonal projections
such that $\sum_{n \in \GG} P_n = Id$,  with
the maximum rank of $P_n$ being finite,  thus
$$
\displaystyle{\Hi = \bigoplus_{n\in \GG} Ran(P_n).}
$$
Let $h_0$ denote a bounded self-adjoint operator 
on $\Hi$ and consider the random operator, we stated in equation (\ref{disc}),
\begin{equation}\label{eq:randOp}
 h^\omega=h_0+\sum_{n \in \Lambda_\infty} \omega_n P_n,
\end{equation}
where the random variables $\omega_n$ satisfies Hypothesis \ref{hyp:potential} below.
Given a finite subset $\Lambda\subset\GG$, we will denote $P_{\Lambda}=\sum_{n\in \Lambda} P_n$, $\Hi_\Lambda= P_{\Lambda}\Hi$ and
\begin{equation}\label{eq:cutOffRandOp}
 h^\omega_\Lambda=P_{\Lambda} h^\omega P_{\Lambda}
\end{equation}
denotes the restriction of $h^\omega$ to $\Hi_\Lambda$.

We have the following assumptions on the quantities involved in the model.
\begin{hypothesis}\label{hyp:potential}
We assume that the random variables $\omega_n$ are independent and
distributed according to a density $\rho_n$ which are compactly supported
in $(0, 1)$and satisfy $\rho_n \in C^m((0, 1))$ for some $m \in \NN$ and
\begin{align}\label{eq:potentialassm}
\dd = \sup_{n} \max_{\ell \leq m} \|\rho_n^{(\ell)}\|_\infty < \infty.
\end{align}
\end{hypothesis}
We note that as long as $\rho_n\in C^m((a,b))$ for some $-\infty<a<b<\infty$,  
a scaling and translation will move its support to $(0,1)$. 
So our support condition is no loss of generality.

\begin{hypothesis}\label{hyp:expLocHyp}
A compact interval $J\subset \subset \RR$ is said to be in 
\emph{region of localization} for $h^\omega$ with exponent $0 < s  < 1$ 
and rate of decay $\xi_s>0$, if there exist $C>0$ such that
\begin{equation}\label{eq:locRes}
\sup_{\Re(z) \in J, \Im(z) >0}\Ex_\omega\left[\norm{P_n(h^\omega- z)^{-1}P_k}^s\right]\leq C e^{-\xi_s d(n,k)}
\end{equation}
for any $n,k \in \GG$.  For the operators $h_{\Lambda_K}^\omega$
exponential localization is defined with 
$\Lambda, h^\omega_{\Lambda_K},\xi_{s,\Lambda_K}$ replacing
$\GG,h^\omega,\xi_s$ respectively in the above bound. 

{\it We assume  
that for our models, for all $\Lambda_K$, with $K \geq N$ the inequality
(\ref{eq:locRes}) holds for some $\xi_s >0$ and  
$\xi_{s, \Lambda_K} \geq \xi_{s}$, for all $\Lambda_{K}$ with $K \geq N$.
We also assume that the constants $C, \xi_s$ do not change if we change
the distribution $\rho_n$ with one of its derivatives at finitely many sites $n$.} 
\end{hypothesis}

\begin{remark}\label{changerhodisc}
For large disorder models one can get explicit values for $\xi_s$
from the papers of Aizenman-Molchanov \cite{MR1244867} or 
Aizenman \cite{MR1301371}.
For example the Anderson model on $\ell^2(\ZZ^d)$ with disorder parameter 
$\lambda >> 1$, typically $\xi_s = -s\ln\frac{C_{s,\rho} 2d }{\lambda}$, 
for some constant $C_{s,\rho} < \infty$ 
that depends on the single-site density $\rho$ and is independent of 
$\Lambda$. So $\xi_{s,\Lambda} = \xi_s >0$ for large enough $\lambda$.  
Similarly for the Bethe lattice with connectivity $K+1 > 1$, 
$\xi_{s, \Lambda} = \xi_s = - s \ln\frac{C_{s,\rho} (K+1)}{\lambda}$.   
Going through Lemma 2.1 their paper, and tracing through the constants,
we see that our assumption about changing the distribution at finitely 
many sites is valid.
\end{remark}

Henceforth let $E_A(\cdot)$ denote the projection valued spectral measure of
a self-adjoint operator $A$.
Our main goal in this section is to show that
$$\etan(E)=\Ex_\omega\left[tr(P_0E_{h^\omega}(-\infty,E))\right]\qquad $$
is $m$ times differentiable in the region of localization,
if $\rho$ has a bit more than $m$ derivatives,
which means that the density of states DOS is $m-1$ times differentiable.
Our theorem is the following, where we tacitly assume that the spectrum
$\sigma(h^\omega)$ is a constant set, a fact proved by Pastur \cite{Pastur}
for a large class of random self-adjoint operators.

\begin{theorem}\label{thm:smoothnessThmDiscreteCase}
Consider the random self-adjoint
operators $h^\omega$ given in equation \eqref{eq:randOp} 
on the Hilbert space $\Hi$ and a graph $\GG$ satisfying the condition \eqref{growthofG} with $g(N) = N^\alpha$, for some $\alpha >0$.
We assume that $\omega_n$ is distributed with density $\rho_n$ satisfying
the Hypothesis \ref{hyp:potential} and, with $m$ as in the Hypothesis,  
$\rho_n^{(m)}$ is $\tau$-H\"older continuous for some $ 0 < \tau < 1$.
Assume that $J$ is an interval in the region of localization for which 
the Hypothesis \ref{hyp:expLocHyp} hold for some $0 < s < \tau$.  
Then the function
\begin{equation}\label{eq:IDS}
\etan(E)=\Ex_\omega\left[ tr(P_0 E_{h^\omega}(-\infty,E))\right] \in C^{m-1}(J)
\end{equation}
and $\etan^{(m)}(E)$ exists a.e. $E \in J$.
\end{theorem}

\begin{remark} 
\begin{enumerate}
 \item We stated the Theorem in this generality so that it applies to multiple models, such as the
Anderson models on $\ZZ^d$ other lattice or graphs, having the property
that the number of points at a distance $N$ from any fixed point grow 
polynomially in $N$.  The models for which
this Theorem is valid also include
higher rank Anderson models, long range hopping with some 
restrictions, models with off-diagonal disorder to state a few.  In all of these models, by including 
sufficiently high diagonal disorder, through a coupling constant $\lambda$ on 
the diagonal part, we will have 
exponential localization for the corresponding operators via the 
Aizenman-Molchanov method.  
So this Theorem  gives the Regularity of DOS in all such models.  
For the Bethe lattice and other countable sets
for which $g(N)$ is like $\ln(N)$,
our results hold but the smoothness $m$ that can be obtained 
is restricted by the localization length by a condition such as 
$\xi_s > m \ln K$.  So in this work we do not consider such type of setting.

\item This Theorem also gives smoothness of DOS in the region of localization
 for the intermediate disorder cases
considered for example by Aizenman \cite{MR1301371} who exhibited exponential localization for such models in part of the spectrum.

\item In the case $h^\omega$ is not the Anderson model, all these results 
are new and it is not clear that the method
of proof using super symmetry, as done for the Anderson model at high disorder,
will even work for these models.

\item We note that in the proof we will take at most $m-1$ derivatives
of resolvent kernels in the upper half-plane and show their boundedness,
but we have a condition that the function $\rho$ has a $\tau$-H\"older
continuous derivative.  The extra $1+\tau$ `derivatives' are needed for 
applying the Theorem \ref{thm:remarkable} in the inequality (\ref{eq:New343})
from the inequality (\ref{eq:New341}).
\end{enumerate}

\end{remark}

\begin{proof} 
Since the orthogonal projection $P_0$ is finite rank, we can write 
$P_0 = \sum_{i=1}^r |\phi_i\rangle\langle\phi_i|$ using a set 
$\{\phi_i\}$ of finitely many orthonormal vectors.
Then we have, 
$$
\etan(E) = \sum_{i=1}^r \Ex_\omega\left(\langle \phi_i, E_{h^\omega}((-\infty, E)) \phi_i \rangle\right).
$$
The densities of the measures 
$\langle \phi_i, E_{h^\omega}(\cdot) \phi_i \rangle$ are bounded by 
Lemma \ref{specavg} for each $i=1, \dots, r$.  Hence $\etan$ is differentiable 
almost everywhere and its derivative, almost everywhere, 
is given by the boundary values,
$$
\frac{1}{\pi} \Ex_\omega \bigg( tr\left(P_0\Im(h^\omega - E - i0)^{-1}\right) \bigg)
$$ 
is bounded.  
The Theorem follows from Lemma \ref{derivatives} once we show
\begin{equation}\label{toshowdisc}
\sup_{\Re(z) \in J, \Im(z) >0} \frac{d^\ell}{dz^\ell} \Ex_\omega \left[tr(P_0(h^\omega-z)^{-1})\right] < \infty,
\end{equation}
for all $\ell \leq m-1$, since such a bound implies that $m-1$ derivatives
of $\eta$ are continuous and its $m$th derivative exists almost everywhere, 
since $h^\omega$  are bounded operators.  
The projection $P_0$ is finite rank
which implies that the bounded operator valued analytic functions  
$P_0(h^\omega - z)^{-1}, P_0(h_\Lambda^\omega - z)^{-1}$ are 
trace class for $z \in \CC^+$.   
Therefore the linearity of the trace and the dominated convergence 
theorem together imply that 
\begin{equation}\label{eq:New310}
\Ex_\omega \left[tr(P_0(h^\omega-z)^{-1}-P_0(h^\omega_{\Lambda}-z)^{-1})\right]\xrightarrow{\Lambda\rightarrow\GG}0,
\end{equation} 
compact uniformly in $\CC^+$.  For the 
rest of the proof we set $h_K^\omega = h_{\Lambda_K}^\omega$  for ease
of writing.

The convergence given in equation (\ref{eq:New310}) 
implies that the telescoping sum,
\begin{align*}
&\Ex_\omega\left[ tr(P_0(h^\omega_{M}-z)^{-1})\right] 
\\ &= \sum_{K=N}^{M} \bigg(\Ex_\omega\left[ tr(P_0(h^\omega_{K+1}-z)^{-1})\right]
-\Ex_\omega\left[ tr(P_0(h^\omega_K-z)^{-1})\right]\bigg)
\\ & ~~~~ +\Ex_\omega\left[ tr(P_0(h^\omega_N-z)^{-1})\right]
\end{align*}
also converges compact uniformly, in $\CC^+$ to 
$$
\Ex_\omega \left[tr(P_0(h^\omega-z)^{-1})\right],
$$
which implies that their derivatives of all orders also converge compact
uniformly in $\CC^+$.  

Therefore the inequality (\ref{toshowdisc}) follows
if we prove the following uniform bound, for all $ 0 \leq \ell \leq m-1$ 
and $N$ large,
\begin{equation}\label{eq:thm1pfEq1}
\sum_{K=N}^{\infty}\displaystyle{\sup_{\Re(z) \in J} \left|\frac{d^\ell}{dz^\ell}\bigg(\Ex_\omega\left[ tr(P_0(h^\omega_{K+1}-z)^{-1})
-\Ex_\omega\left[ tr(P_0(h^\omega_K-z)^{-1})\right]\bigg)\right]\right|} < \infty.
\end{equation}

To this end  we only need to estimate
\begin{equation}\label{eq:thm1pfEq7}
\bigg|\frac{d^l}{dz^l}\Ex_\omega\left[ tr(P_0(h^\omega_{K+1}-z)^{-1}P_0)-tr(P_0(h^\omega_K-z)^{-1}P_0)\right]\bigg|
\end{equation}
for $\Re(z) \in J$ where we used the trace property to get an extra $P_0$
on the right and set 
$G^\omega_{M}(z)=P_0(h^\omega_{M}-z)^{-1}P_0, ~~ M \in \NN$
for further calculations. 

The function 
$$
f_\epsilon(\vec{\omega}-E\vec{1}) = tr(G^\omega_{K}(E+i\epsilon)) 
$$
is a complex valued bounded measurable function on $\RR^{K+1}$ for each
fixed $\epsilon >0$.  Therefore we compute the derivatives in 
$E$ of its expectation
$$
h_\epsilon(E) = \Ex_\omega \big(f_\epsilon(\vec{\omega}-E\vec{1})\big) 
= \EE\big( tr(G^\omega_{K}(E+i\epsilon)^{-1})\big) 
$$
using Lemma \ref{thm:finitesmooth}.  This caculation gives in the
notation of that Lemma,
\begin{equation}\label{eq:New311}
\frac{d^\ell}{dE^\ell} \Ex_\omega\big( tr(G^\omega_K(E+i\epsilon))\big)
= \int tr(G^\omega_K( E + i\epsilon))  {\bf D}^\ell \Phi_K(\vec{\omega})d\vec{\omega},
\end{equation}
where we set 
$\displaystyle \Phi_K(\vec{\omega}) = \prod_{n \in \Lambda_K} \rho_n(\omega_n),
~ d\vec{\omega} = \prod_{n \in \Lambda_K} d\omega_n$.

It is not hard to see that for each $0 \leq \ell \leq m-1$,
\begin{equation}\label{eq:New312}
\int tr(G^\omega_K(E + i\epsilon))  {\bf D}^\ell \Phi_K(\vec{\omega})d\vec{\omega},
=\int tr(G^\omega_K( E + i\epsilon))  {\bf D}^\ell \Phi_{K+1}(\vec{\omega})d\vec{\omega},
\end{equation}
since the integrand is independent of $\omega_n, n \in \Lambda_{K+1}\setminus
\Lambda_K$ and $\rho_n$ satisfies $\int \rho_n^{(j)}(x)dx = \delta_{j0}$. 
We set
\begin{equation}\label{eq:New315}
R(\omega, K, E, \epsilon) =  tr\left(G^\omega_{K+1}(E+i\epsilon) - G^\omega_K(E+i\epsilon)\right)
\end{equation}
to simplify writing.  We may write the argument $\omega$ of $R(\omega, K, E, \epsilon)$ below
in terms of the vector notation $\vec{\omega}$ for uniformity as it is
a function of the variables $\{\omega_n, n \in \Lambda_{K+1}\}$.

Then combining the equations (\ref{eq:New311}, \ref{eq:New312}) inside the
absolute value of the expression in equation (\ref{eq:thm1pfEq7}) to be 
estimated we have to consider the quantity, for $K \geq N$,
\begin{align}\label{eq:New313}
&T_{K,\ell}(E,\epsilon) = \frac{d^\ell}{dE^\ell} \Ex_\omega\left[ tr\big(G^\omega_{K+1}(E+i\epsilon) - G^\omega_{K}(E+i\epsilon)\big)\right] \nonumber \\
&= \int_{\RR^{K+1}} R(\vec{\omega}, K, E, \epsilon) ({\bf D}^\ell \Phi_{K+1})(\vec{\omega})d\vec{\omega}.
\end{align} 
To prove the theorem we need to show that
\begin{equation}\label{eq:New314}
\sum_{K=N}^\infty \sup_{E \in J, \epsilon >0} |T_{K, \ell}(E, \epsilon)| < \infty.
\end{equation}
Multinomial expansion of 
$\displaystyle{{\bf D}^\ell = \bigg( \sum_{n \in \Lambda_{K+1}} \frac{\partial}{\partial \omega_n} }\bigg)^\ell$ gives the relation
\begin{align}\label{eq:New325}
&T_{K,\ell}(E, \epsilon)  \nonumber \\
&=\displaystyle{ \sum_{\substack{k_0+\dots+k_{K} = \ell\\ k_n \geq 0}} \big(\substack{\ell \\ k_0, \dots, k_{K}}\big)\int_{\RR^{K+1}} R(\vec{\omega}, K, E, \epsilon) \bigg(\prod_{n=0}^{K+1}\frac{\partial^{k_n}}{{\partial^{k_n}\omega_{n}}} \rho_n(\omega_{n})d\omega_{n}\bigg)}.
\end{align} 
We use Fubini to interchange the trace and an integral over $\omega_0$ to
get 
\begin{align}\label{eq:New326}
&T_{K,\ell}(E, \epsilon) \nonumber \\
&=\displaystyle{\sum_{\substack{k_0+\dots+k_{K} = \ell\\ k_n \geq 0}} \big(\substack{\ell \\ k_0, \dots, k_{K}}\big)} \int_{\RR^{K}} tr\bigg(\int \big(G_{K+1}^\omega(E+i\epsilon) - G_{K}^\omega(E+i\epsilon)\big){\rho_0^{(k_0)}(\omega_0) d\omega_0}\bigg) \nonumber \\
&\times \bigg(\prod_{n\in \Lambda_{K+1}, n \neq 0} \rho_n^{(k_n)}(\omega_{n})d\omega_{n}\bigg) .  
\end{align}
Then taking absolute value of $T$ and estimate the $\omega_0$ integrals
using the Theorem \ref{thm:remarkable}, displaying explicitly the dependence
on the $\rho$ or its derivatives in the constant $\Xi$ appearing 
in that theorem, to get, for $0 < s < 1/2$ (the choice for $s$ will
become clear in Lemma \ref{lem:New301}),   
\begin{align}\label{eq:New341}
&T_{K,\ell}(E, \epsilon) \leq \nonumber \\
&\qquad \displaystyle{\sum_{\substack{k_0+\dots+k_{K} = \ell\\ k_n \geq 0}} \big(\substack{\ell \\ k_0, \dots, k_{K}}\big)}
\Xi(\rho_0^{(k_0)}) tr(P_0)\nonumber \\
&\qquad\times  \int_{\RR^{K}} \bigg(\int \|\big(G_{K+1}^\omega(E+i\epsilon) - G_{K}^\omega(E+i\epsilon)\big)\|^s \phi_R(\omega_0) d\omega_0\bigg) \nonumber \\
&\qquad\qquad\times \big(\prod_{n\in \Lambda_{K+1}, n \neq 0} |\rho_n^{(k_n)}(\omega_{n})|d\omega_{n}\big) .  
\end{align}
We set
\begin{align}\label{eq:New735}
\tilde{\rho_n} = \frac{\rho_n^{(k_n)}}{\|\rho_n^{(k_n)}\|_1}, ~~ n \neq 0, ~~
\tilde{\rho_0} = \frac{\phi_R}{\|\phi_R\|_1} 
\end{align}
and set, using the inequality (\ref{eq:potentialassm}), 
$C_0=\max\{\dd, \|\phi_R\|_1\}$, where $\dd$ is such that 
$\|\rho^{(k_n)}_n\|_1\leq \dd~\forall~n \neq 0$. 
We note here that at most $\ell$ of $\tilde{\rho}_n$ differ from 
$\rho_n$ itself and that $\|\rho_n\|_1 = 1$. We then we get the bound
\begin{align}\label{eq:New343}
&|T_{K,\ell}(E, \epsilon)|  \nonumber \\
&\qquad\leq C_0^\ell \sup_{j\leq \ell} \Xi(\rho^{(j)}) tr(P_0) 
\sum_{\substack{k_0+\dots+k_{K} = \ell\\ k_n \geq 0}} \big(\substack{\ell \\ k_0, \dots, k_{K}}\big)\nonumber\\
& \qquad\qquad\times \int_{\RR^{K+1}} \|\big(G_{K+1}^\omega(E+i\epsilon) - G_{K}^\omega(E+i\epsilon)\big)\|^s \prod_{n \in \Lambda_{K+1}}  \tilde{\rho}_{n}(\omega_{n})d\omega_{n}
\end{align}
We denote the probability measure 
$$
d\mathbb{P}_{(k)}(\vec{\omega}) = \prod_{n \in \Lambda_{K+1}}  \tilde{\rho}_{n}(\omega_{n})d\omega_{n},
$$ 
and expectation as $\EE_{\mathbb{P}_{(k)}}$, since it does not 
depend on the indices.  We also set, 
$$
 C_{1,m}=\sup_{0\leq \ell\leq m}\{C_0^\ell\},~ C_{2, m} = 
\sup_{n \in \GG, j\leq m} \{\Xi(\rho_n^{(j)})\}.
$$
Then the inequality (\ref{eq:New343}) becomes
\begin{align}\label{eq:New344}
|T_{K,\ell}(E, \epsilon)| &\leq C_{1,m} C_{2,m}tr(P_0) \displaystyle{ \sum_{k_{0}+\dots+k_{K} = \ell} \big(\substack{\ell \\ k_{0}, \dots, k_{K}}\big)} \nonumber \\
&\qquad \times \EE_{\mathbb{P}_{(k)}} \bigg[\|\big(G_{K+1}^\omega(E+i\epsilon) - G_{K}^\omega(E+i\epsilon)\big)\|^s \bigg].
\end{align}
We use the estimate for the expectation $\EE_{K+1}(\cdot)$ from 
Lemma \ref{lem:New301} to get the following bound, for some constant 
$C_6$ independent of $K$,
\begin{align}
\sup_{ E \in J, \epsilon >0} |T_{K,\ell}(E, \epsilon)| 
& \leq C_6C_5 \displaystyle{ \sum_{\substack{k_{0}+\dots+k_{K} = \ell}}} \big(\substack{\ell \\ k_{0}, \dots, k_{K}}\big) 
(1+2K) e^{-\xi_{2s} K^\alpha} \nonumber \\
& \leq C_6 C_5 (K+1)^\ell (1+2K) e^{-\xi_{2s} |K|^\alpha}.
\end{align}
From this bound
 the summability stated in the inequality (\ref{eq:New314}) 
follows since we assumed that $\xi_{2s} >0$, completing the proof 
of the Theorem. 
\end{proof}
We needed the exponential bound on the resolvent estimate, which is the focus of the following lemma.

\begin{lemma}\label{lem:New301}
We take the interval $J$ stated in Theorem \ref{thm:smoothnessThmDiscreteCase}, then we have the bound
\begin{align*}
&\sup_{\Re(z) \in J, \Im(z) >0} \EE_{\mathbb{P}_{(k)}} \bigg[\big\|\big(G_{K+1}^\omega(z) - G_{K}^\omega(z)\big)\big\|^s \bigg]
\\ 
&\qquad\qquad\leq C_5(m,Rank(P_0), \dd, h_0, R,s) (2K+1) e^{-\xi_{2s} K^\alpha}. 
\end{align*}
\end{lemma}
\pf We start with the resolvent identity
\begin{align*}
G_{K}^\omega (z) - G_{K+1}^\omega (z) &=P_0 (h^\omega_{K} - z)^{-1}  - (h_{K+1}^\omega - z)^{-1} P_0\\
   &=P_0 (h^\omega_{K} - z)^{-1}[h^\omega_{K+1}-h^\omega_{K}] (h_{K+1}^\omega - z)^{-1} P_0\\
   &=P_0 (h^\omega_{K+1} - z)^{-1}P_{\Lambda_K}h_0P_{K+1}(h_{K+1}^\omega - z)^{-1} P_0.
\end{align*}
Now we write 
\begin{align}
\label{eq:nw351}
\|\big(G_{K+1}^\omega (z) - G_{K}^\omega (z)\big)\|^s & \leq \|h_0\|^s\|P_{K+1}(h_{K+1}^\omega - z)^{-1} P_0\|^s
\|P_0 (h^\omega_{K+1} - z)^{-1}P_{\Lambda_K}\|^s \nonumber \\
&\leq \|h_0\|^s\|P_{K+1}(h_{K+1}^\omega - z)^{-1} P_0\|^s \nonumber \\
& \qquad \times \sum_{n\in\Lambda_K} \|P_0 (h^\omega_{K+1} - z)^{-1}P_n\|^s. 
\end{align}
We estimate the last line first by expanding $P_{\Lambda_K} = \sum_{n\in\Lambda_K}  P_n$ and estimate the norms 
(using $\|v\|_1\leq\|v\|_s $ for $ 0 < s <1$), then interchange the sum  
and the expectation and finally use Cauchy-Schwartz inequality to get the bound 
\begin{align}
\label{eq:nw352}
& \EE_{\mathbb{P}_{(k)}} \big(\|\big(G_{K+1}^\omega (z) - G_{K}^\omega (z)\big)\|^s\big)\nonumber \\ 
& \qquad\leq \|h_0\|^s\sum_{n\in\Lambda_K}\big(\EE_{\mathbb{P}_{(k)}}\big(\|P_{K+1}(h_{K+1}^\omega - z)^{-1} P_0\|^{2s}\big)\big)^{\frac{1}{2}}\nonumber\\
&\qquad\qquad\qquad\big(\EE_{\mathbb{P}_{(k)}}\big(\|P_0(h_{K+1}^\omega - z)^{-1} P_n \|^{2s}\big)\big)^{\frac{1}{2}}.
\end{align}
We now estimate the above terms by getting an exponential
decay bound for the term with operators kernels of the form $P_{K+1}[\cdot]P_0$
while the remaining factors are uniformly bounded with
the bound independent of $K$, by using the Hypothesis \ref{hyp:expLocHyp}.

Applying the bound on the fractional moments given in the Hypothesis \ref{hyp:expLocHyp},  inequality \ref{eq:locRes} we get 
\begin{align*}
&\EE_{K+1} \big( \|P_{n}(h_K^\omega - z)^{-1}P_0\|^{2s}\big) \leq C , ~~ n \in \Lambda_{K}, \nonumber \\
&\EE_{K+1} \big(\|P_0 (h^\omega_{K+1} - z)^{-1}P_{n}\|^{2s}\big) \leq C , ~~ n \in \Lambda_K \nonumber \\
&\EE_{K+1} \big(\|P_0 (h^\omega_{K+1} - z)^{-1}P_{K+1}\|^{2s}\big) \leq C e^{-\xi_s K^\alpha}, \\
&\EE_{K+1} \big(\|P_{K+1} (h^\omega_{K+1} - z)^{-1}P_{0}\|^{2s}\big) \leq C e^{-\xi_s K^\alpha}.
\end{align*}
Using these bounds in the inequality (\ref{eq:nw352}), we get 
the bound (after noting that the sum has $2K$ terms,
so we get $(1+2K)$ as the only $K$ dependence
other than the exponential decay factor), 
\begin{align*}
&\leq C_5(m,Rank(P_0), \dd, h_0, R,s) (1 + 2K) e^{-\xi_{2s} K^\alpha},  
\end{align*}
which is the required estimate to complete the proof of the Lemma. \qed

\section{The Continuous case}
In this section we show that the density of states of some Random Schr\"odinger
operators are almost as smooth as the single site distribution.  
On the Hilbert space $L^2(\RR^d)$ we consider the operator
$$H_0=\sum_{i=1}^d \left(-i\frac{\partial}{\partial x_i}+A_i(x)\right)^2,$$
with the vector potential $\vec{A}(x)=(A_1(x),\cdots,A_d(x))$
assumed to have sufficient regularity so that $H_0$
is essentially self-adjoint on $C_0^\infty(\RR^d)$.

The random operators considered here are given by
\begin{equation}\label{eq:ranOpCont}
H^\omega=H_0+\lambda\sum_{n\in\ZZ^d} \omega_n u_n,
\end{equation}
where $\{\omega_n\}_{n\in\ZZ^d}$ are independent real random variables satisfying Hypothesis \ref{hyp:potential},
$u_n$ are operators of multiplication by the functions $u(x-n)$, for $n \in \ZZ^d$ and $\lambda >0$ a coupling constant.

We have the following hypotheses on the operators considered above
to ensure $H^\omega$ continue to be essentially self-adjoint
on $C_0^\infty(\RR^d)$ for all $\omega$.
By now it is well known
in the literature (see for example the book of Carmona-Lacroix \cite{MR1102675})
that the spectral and other functions of these operators we consider below
will have the measurability properties, as functions of $\omega$,
required for the computations we perform on them and we will
not comment further on measurability.

\begin{hypothesis}\label{hyp:potentialCtsCase}
\begin{enumerate}
\item The random variable $\{\omega_n\}_n$ satisfies the Hypothesis \ref{hyp:potential}.
\item The function $0\leq u\leq 1$ is a non-negative smooth function on $\RR^d$
such that for some $0 < \epsilon_2 < \half, 0 < \epsilon_1 < 1$, it satisfies
\begin{eqnarray*}
&u(x) = \begin{cases} 0, ~~  x \notin (-\frac{1}{2}-\epsilon_1,\frac{1}{2}+\epsilon_1)^d \\
1, ~~ x \in  (-\frac{1}{2}+\epsilon_2,\frac{1}{2}-\epsilon_2)^d
\end{cases} \\
&\displaystyle{\sum_{n\in\ZZ^d}u(x-n)=1}\qquad x\in\RR^d.
\end{eqnarray*}
\end{enumerate}
\end{hypothesis}

We need some notation before we state our results.
 Given a subset $\Lambda\subset\ZZ^d$, we set

\begin{equation}\label{eq:New000}
[\Lambda]=\bigg\{x\in\RR^d: \sum_{n\in\Lambda}u(x-n)=1\bigg\}
\end{equation}
and denote the restrictions of $H_0, H^\omega$ to $[\Lambda]$
respectively by $H_{0,\Lambda}, H^\omega_\Lambda$. 
As an abuse of notation, whenever we talk about restricting the operator on $\Lambda$, we will mean restriction onto $[\Lambda]$.
We need this distinction because $\sum_{n\in\Lambda}u(x-n)= 1$ only on  $[\Lambda]$ and we need the complete covering condition.
While the
boundary conditions are not that important, we will work with
 Dirichlet boundary conditions in this section.
We will also denote $u_{n,\Lambda}$ to be the restriction of $u_n$
to $[\Lambda]$ when the need arises.  We denote by $E_A(\cdot)$
the projection valued spectral measure of a self-adjoint operator $A$
and from the context it will be clear that this symbol will not
be confused with points in the spectrum denoted by $E$.
We denote the Integrated Density of States (IDS) by
\begin{equation}\label{eq:contIDS}
\etan_\Lambda(E)=\Ex_\omega\left[ tr(u_0 E_{H^\omega_\Lambda}(-\infty,E])\right]\qquad for~E\in\RR,
\end{equation}
and the subscript $\Lambda$ on the IDS is dropped in the case of
the operator $H^\omega$. 

We start with our Hypothesis on the localization. where we set
$P_n$ to be the orthogonal projection onto $L^2(supp(u_n))$.

\begin{hypothesis}\label{hyp:expLocHypCts}
A compact interval $J\subset\RR$ is said to be in \emph{the region of localization} for $H^\omega$ with rate of decay $\xi_s$ and
exponent $0 < s <1$,  if there exists $C,\xi_s>0$ such that
\begin{equation}\label{eq:locResCont}
\sup_{\Re(z) \in J, \Im(z) >0}\Ex_\omega\left[\norm{P_n(H^\omega- z)^{-1}P_k}^s\right]\leq C e^{-\xi_s \norm{n - k}}
\end{equation}
for any $n,k \in \ZZ^d$.  For the operators $H_\Lambda^\omega$
exponential localization is similarly defined with $\Lambda, H^\omega_\Lambda, \xi_{s,\Lambda}$ replacing $\ZZ^d, H^\omega,\xi_s$ respectively in the bound for the same $J$.

{\it We assume that for all $\Lambda$ large enough  
$\xi_{s,\Lambda} \geq \xi_s$ for $J$ in the region of localization 
and the constants $C, \xi_s$ do not change if we change the density 
$\rho_n$ with one of its derivatives at finitely many $n$.} 
\end{hypothesis}

\begin{remark}\label{rem:changerhocont}
We note that the above Hypothesis holds with $\xi_s >0$, 
for the models of the type we
consider under a {\it large disorder} condition, introduced via a 
coupling constant.  The condition $\xi_s >0$ is sufficient for our Theorem
and there is no need to specify how large it should be. Similarly
the multiscale analysis which is the starting point of the fractional
moment bounds, uses apriori bounds that depend on the Wegner estimate which
depends on only the constant $\dd$. So changing the distribution $\rho_n$
with one of its derivatives at finitely many points $n$ does not affect the
constants $C, \xi_s$.
\end{remark}
 
Our main Theorem given next, is the analogue of 
the Theorem \ref{thm:smoothnessThmDiscreteCase}.
We already know from Lemma \ref{lem:traceclass},
that $u_0E_{H^\omega}(-\infty,E)$ is trace class for any $E \in \RR$,
hence we will be working with
\begin{equation}\label{eq:contIDSFull}
 \etan(E)=\Ex_\omega\left[tr(u_0 E_{H^\omega}(-\infty,E))\right]\qquad for~ E\in\RR,
\end{equation}
The function $\etan$ is well defined by Lemma \ref{lem:traceclass} and is
known to be continuous (see \cite[Theorem 1.1]{MR2362242} for
example)  whenever $\rho$ is continuous.

By the Pastur-Shubin trace formula for the IDS, the function $\etan$ is
at most a constant multiple of IDS, since $\int u_0(x) dx $ may not
be equal to 1, but this discrepancy does not affect the smoothness
properties, so we will refer to $\etan$ as the IDS below.

Our main Theorem given below implies that the density of states
DOS is $m-1$ times differentiable in $J$ when $\rho$ satisfies the conditions
of the Theorem.

\begin{theorem}\label{thm:smoothnessThmContCase}
On the Hilbert space $L^2(\RR^d)$ consider the self-adjoint operators
 $H^\omega$ given by \eqref{eq:ranOpCont}, satisfying the Hypothesis \ref{hyp:potentialCtsCase}. 
Let $J$ be an interval in the region of localization
satisfying the Hypothesis \ref{hyp:expLocHypCts} 
with $\xi_s > 0$ for some  $0 < s < 1/6$.
Suppose the density $\rho \in C_c^m((0,\infty))$,
and $\rho^{(m)}$ is $\tau$-H\"older continuous for some $ s < \tau/2$.
Then $\etan \in C^{(m-1)}(J)$ and $\etan^{(m)}$ exists
almost everywhere in $J$.
\end{theorem}
\begin{remark}\label{rem:New101}
A Theorem of Aizenman et.al. \cite[Theorem 5.2]{MR2207021}
shows that there are operators $H^\omega$ of the type we consider 
for which the Hypothesis \ref{hyp:expLocHypCts} is valid for 
large coupling $\lambda$, where it was required that $0 < s < 1/3$.
We take $0 < s < 1/6$ as we need to controls $2s$-th moment of
averages of norms of resolvent kernels in our proof. 
\end{remark}
\begin{proof}
We consider the boxes $ \Lambda_L=\{-L,\cdots,L\}^d$,
and set $H_L^\omega = H_{\Lambda_L}^\omega , ~~
\etan_L = \etan_{\Lambda_L}^\omega$.

The strong resolvent convergence of
 $H_{\Lambda_L}^\omega$ to $H^\omega$, which is easy to verify,
implies that $\etan_{\Lambda_L}$ converges to $\etan$ point wise
since $\etan$ is known to be a continuous function for the operators
we consider.  Since $tr(u_0E_{H^\omega_L}((-\infty, E]))$ 
is a bounded measurable complex valued function, $\etan_L \in C^m(J)$,
by Theorem \ref{thm:finitesmooth}. 
Therefore it is enough to show that
$\etan(\cdot)-\etan_{\Lambda_N}(\cdot)$ (which is a difference
of distribution functions of the $\sigma$-finite measures
$tr(u_0E_{H^\omega}(\cdot))$ and $tr(u_0 E_{H_N^\omega}(\cdot))$ appropriately
normalized) is in $C^m(J)$ for some $N$. 
We will need to use the Borel-Stieltjes transforms of these  measures
for the rest of the proof, but these transforms are not defined 
because $u_0(H^\omega_N-z)^{-1}$
fails to be in trace class.  Therefore we have to approximate $u_0$
using finite rank operators first.

To this end let $Q_k$ be a sequence of finite rank orthogonal projections,
in the range of $u_0$ such that they converge to the identity on this range.
We then define,
\begin{align}\label{eq:New701}
 \etan_{L, Q_k}(E) = \Ex_\omega \left(tr(Q_ku_0E_{H^\omega_L}(-\infty, E]) \right).
\end{align}
Since the projections $Q_k$ strongly converge to the identity on the range
of $u_0$, the projections $Q_ku_0E_{H_{L}^\omega}((-\infty, E))$ also converge
strongly to $u_0E_{H_{L}^\omega}((-\infty, E))$ point wise in $E$.  This
convergence implies that $\etan_{L, Q_k}(E)$ converge point wise to $\etan_L(E)$
for any fixed $L$. Henceforth we drop the subscript on $Q_k$ but remember
that the rank of $Q$ is finite.

Since $Q$ is finite rank,  the measures $tr(Qu_0E_{H^\omega_L}(\cdot))$ are finite
measures.  Therefore we can define the Borel-Stieltjes transform of the finite
signed measure
$$
\Ex_\omega\left[tr(Qu_0E_{H^\omega_{L+1}}(\cdot))- tr(Qu_0E_{H^\omega_L}(\cdot))\right],
$$
namely
\begin{align}\label{eq:New702}
&\Ex_\omega \left[ tr(Qu_0(H_{L+1}^\omega - z)^{-1} - tr(Qu_0(H_L^\omega -z)^{-1})\right]\nonumber \\
&\qquad= \int \frac{1}{x-z}~~
d\Ex_\omega\left[ tr(Qu_0E_{H^\omega_{L+1}}(x))- tr(Qu_0E_{H^\omega_L}(x))\right],
\end{align}
where the signed measure has finite total variation for each $Q$ and each $L$.
Then the derivative of $\etan_{L+1, Q}(E) - \etan_{L, Q}(E)$ are given by
\begin{align}\label{eq:New703}
&\lim_{\epsilon \downarrow 0}
\frac{1}{\pi} \Ex_\omega \left[ tr(Qu_0\Im(H_{L+1}^\omega -E- i \epsilon)^{-1}) - tr(Qu_0\Im(H_L^\omega - E- i \epsilon)^{-1})\right]\nonumber \\
&=\lim_{\epsilon \downarrow 0}
\frac{1}{\pi} \Ex_\omega \left[ tr\left[ Qu_0\bigg(\Im(H_{L+1}^\omega -E- i \epsilon)^{-1} - \Im(H_L^\omega - E- i \epsilon)^{-1}\bigg) \right ] \right].
\end{align}
Then, using the idea of a telescoping sum, as done in the previous 
section, we need to prove that
\begin{align}\label{eq:New704}
\sum_{L=N}^{\infty}\displaystyle{\sup_{\Re(z) \in J} \left|\frac{d^\ell}{dz^\ell}\bigg(\Ex_\omega\left[ tr(Qu_0(H^\omega_{L+1}-z)^{-1})
-\Ex_\omega\left[ tr(u_0(H^\omega_L-z)^{-1})\right]\bigg)\right]\right|} < \infty.
\end{align}
We set (takeing $\kappa(L)$ as the volume of $\Lambda_L\setminus\{0\}$),
\begin{align}\label{eq:New705}
&G_L^\omega(z) = Qu_0 (H_L^\omega - z)^{-1} u_0, ~~~~  
S(\vec{\omega},Q, L, z) = G_{L+1}^\omega(z) - G_{L}^\omega(z), \nonumber \\
&\Phi_{L+1}(\vec{\omega}) = \prod_{n \in \Lambda_{L+1}} \rho(\omega_n), ~~~ \kappa(L) = |\Lambda_L|-1.  
\end{align}
Then, following the sequence of steps leading from equation (\ref{eq:thm1pfEq7}) to equation (\ref{eq:New314}), we need only to consider
\begin{align}\label{eq:New706}
&T(L,\ell,Q, z) = \frac{d^\ell}{dE^\ell} \Ex_\omega\big[ tr(G^\omega_{l+1}(E+i\epsilon) - G^\omega_{l}(E+i\epsilon))\big] \nonumber \\
&= \int_{\RR^{\kappa(L+1)+1}} tr(S(\vec{\omega},Q,L, E, \epsilon)) ({\bf D}^\ell \Phi_{L+1})(\vec{\omega})d\vec{\omega},
\end{align}
to estimate and show that
\begin{align}\label{eq:New707}
\sum_{L=N}^\infty \sup_{\substack{\Re(z) \in J, \Im(z) >0,\\  \ell \leq m,\\  Q}} |T(L,\ell,Q, z)|  < \infty, 
\end{align}
to prove the theorem. Using the steps followed from getting equation (\ref{eq:New326}) from the equality (\ref{eq:New325}), which is an identical calculation here, to get
\begin{align}\label{eq:New708}
&T(L,\ell, Q, z)\nonumber \\
&=\displaystyle{\sum_{\substack{\sum_{n=1}^{\Lambda_{L+1}} k_n = \ell\\ k_n \geq 0}} \big(\substack{\ell \\ k_0, \dots, k_{\kappa(L+1)+1}}\big)} \int_{\RR^{\kappa(L+1)}} tr\bigg(\int \big(G_{L+1}^\omega(z) - G_{L}^\omega(z)\big)\rho_0^{(k_0)}(\omega_0) d\omega_0\bigg) \nonumber \\
&\qquad\qquad \cdot \bigg(\prod_{n\in \Lambda_{L+1}\setminus\{0\}} \rho_n^{(k_n)}(\omega_{n})d\omega_{n}\bigg) .
\end{align}

To proceed further, we need to get a uniform bound in the projection $Q$. 
We will show that the expression 
\begin{align}\label{eq:New801}
{\mathcal G}(L,z,\omega) = u_0 (H_{L+1}^\omega - z)^{-1} - u_0 (H^\omega_L -z)^{-1},
\end{align}
automatically comes with a trace class operator. This fact helps us
drop the $Q$ occurring in the expression
\begin{align}\label{eq:New800}
\big(G_{L+1}^\omega(z) - G_{L}^\omega(z)\big) = Q {\mathcal G}(L, z, \omega) u_0,
\end{align}
making estimates on the trace. 

We need a collection of $d+1$ smooth functions $0 \leq \Theta_j \leq 1, j=0,\dots, d+1$,
where $d$ is the dimension we are working with.  Setting
\begin{align}\label{eq:New802}
\alpha_j = 2^{j+2}, j\in\{0,1,2, \dots, 2d+2\},
\end{align}
we choose the functions $\Theta_j$ from $C^\infty(\RR^d)$ satisfying 
\begin{align}\label{eq:New803}
\Theta_j(x) = \begin{cases} 1, ~~ |x| \leq \alpha_{2j}, \\ 0, ~~ |x| > \alpha_{2j+1},\end{cases} ~~ j=0, \dots d
\end{align}
and note that all the derivatives of $\Theta_j$ are bounded for all $j$,
because they are all continuous and supported in a  compact set. 
These functions satisfy the property
\begin{align}\label{eq:New804}
\Theta_{j+1}\phi = \phi, ~~ if ~ supp(\phi) \subset supp(\Theta_j), ~~ j=0, \dots, d,
\end{align}
in particular 
\begin{equation}\label{eq:New8040}
\Theta_{j+1}\Theta_j = \Theta_j, ~~  for ~~ all ~~ j=0, \dots, d.
\end{equation}
We then take a free resolvent operator 
$R_{L,a}^0 = (H_{0,\Lambda_L} + a)^{-1}$, 
with $a >> 1$. Since, $H_0$ is bounded below, $R_{L,a}^0$ is a bounded positive 
operator for any $L$.  It is a fact that, 
\begin{equation}\label{eq:New805}
[\phi, H_0]R^0_{L,a},R^0_{L,a} u_j \in {\mathcal I}_p, ~~ p > d.
\end{equation}
See Combes et. al. \cite[Lemma 6.1]{MR2362242} and Simon \cite[Chapter 4]{Simon} for further details.
Using the definition of ${\mathcal G}$ given in equation (\ref{eq:New800}),
the relation (\ref{eq:New8040}) and the resolvent equation we get 
\begin{align}\label{eq:New806}
&{\mathcal G}(L,z, \omega) \Theta_0 = u_0\bigg[ (H^\omega_{L+1} - z)^{-1} - (H_L^\omega - z)^{-1}\bigg] \Theta_0 \nonumber \\
&=u_0\bigg[ (H^\omega_{L+1} - z)^{-1} - (H_L^\omega - z)^{-1}\bigg]\Theta_1\Theta_0 \nonumber \\
&=u_0\bigg[ (H^\omega_{L+1} - z)^{-1}\Theta_1 - \Theta_1R_{L,a}^0 + \Theta_1R_{L,a}^0 - (H_L^\omega - z)^{-1}\Theta_1\bigg]\Theta_0 \nonumber \\
&=u_0\bigg[ \big((H^\omega_{L+1} - z)^{-1}\Theta_1 - \Theta_1R_{L,a}^0\big) - \big((H_L^\omega - z)^{-1}\Theta_1 - \Theta_1R_{L,a}^0\big)\bigg]\Theta_1 \Theta_0 \nonumber \\
&=u_0\big[ (H^\omega_{L+1} - z)^{-1}- (H_L^\omega - z)^{-1}\big] \nonumber \\
&\qquad\qquad \cdot \bigg[ [\Theta_1, H_0] + \bigg(z+a - \sum_{|n| \leq \alpha_1} \omega_n u_n\bigg)\Theta_1 \bigg]R_{L,a}^0\Theta_0 \nonumber \\
&={\mathcal G}(L,z,\omega) \bigg[ [\Theta_1, H_0] + \bigg(z+a - \sum_{|n| \leq \alpha_1} \omega_n u_n\bigg)\Theta_1 \bigg] R_{L,a}^0\Theta_0 \nonumber \\
&={\mathcal G}(L,z,\omega) \bigg(A_0(z,a,\alpha_1, H_0) + \sum_{|n| \leq \alpha_1} \omega_n B_{0,n}(a,\alpha_1) \bigg) 
\end{align}
where we used the definition
\begin{align}\label{eq:New807}
&A_0(z,a,\alpha_1, H_0) = \big([\Theta_1, H_0] + (z+a)\Theta_1\big) R_{L,a}^0\Theta_0 \nonumber \\
&B_{0,n}(a,\alpha_1) = -u_n \Theta_1 R_{L,a}^0\Theta_0.
\end{align}
In the above $A_0, B_{0,n}$ are operators independent of $\omega$, each
of which is in ${\mathcal I}_p$, by equation (\ref{eq:New8040}). Using
the definitions and properties of $\Theta_j$,  we see that 
$$
\Theta_2 A_0(z,a,\alpha_1,H_0) = A_0(z,a,\alpha_1,H_0), ~~~ and ~~ 
\Theta_2 B_{0,n}(a,\alpha_1) = B_{0,n}(a,\alpha_1).
$$
Therefore we can repeat this argument by defining for $j=0, \dots d$,
\begin{align}\label{eq:New808}
&A_j(z,a,\alpha_{2j+1}, H_0) = ([\Theta_{2j+1}, H_0] + (z+a)\Theta_{2j+1}) R_{L,a}^0\Theta_{2j} \nonumber \\
&B_{j,n}(a,\alpha_{2j+1}) = -u_n \Theta_{2j+1} R_{L,a}^0\Theta_{2j}, ~ |n| \leq \alpha_{2j+1}, 
\end{align}
by using the fact that 
\begin{align}\label{eq:New809}
&\Theta_{2j} A_{j-1}(z,a,\alpha_{2j+1},H_0) = A_{j-1}(z,a,\alpha_{2j+1},H_0), ~~~ and \nonumber \\
&\Theta_{2j} B_{j-1,n}(a,\alpha_{2j+1}) = B_{j-1,n}(a,\alpha_{2j+1}),
\end{align}
for each $j=1,2,\dots d$.
We can then re-write the equation (\ref{eq:New806}) as
\begin{align}\label{eq:New810}
&{\mathcal G}(L,z, \omega) = {\mathcal G}(L,z, \omega) \prod_{j=0}^{\substack{d \\ \leftarrow }} \bigg(A_j(z,a,\alpha_{2j+1}, H_0) + \sum_{|n| \leq \alpha_{2j+1}} \omega_n B_{j,n}(a,\alpha_{2j+1}) \bigg),
\end{align}
where the arrow on the product indicates an ordered product 
with the operator sum with a lower index $j$ coming to the right of the one with
a higher index $j$. 

Now, counting the number of terms there are in the product, we see that
each sum $\sum_{|n| \leq \alpha_{2j+1}} $ has a maximum of 
$(2 \alpha_{2j+1})^d = 2^{d(2j+4)}$ terms.  A simple computation shows
that there are a maximum of $2^{d^2(d+4)}$ terms, if we 
completely expand out the product.  
In other  words the number of terms are dependent on $d$ but not on $L$.   

We will now write the expression in equation (\ref{eq:New810}) as 
\begin{align}\label{eq:New811}
&{\mathcal G}(L,z, \omega) = \sum_{|n| \leq \alpha_{2d+2}} {\mathcal G}(L,z, \omega) u_n \bigg( \sum_{r_1, r_2 = 0}^{d+1} \omega_n^{r_1} \omega_0^{r_2} P_{n,0}(k, r, \omega)  \bigg),
\end{align}
where $P_{n,0}(k,r)$ is a trace class operator valued function of $\omega$,
but independent of $\omega_0, \omega_n$ for each $k, r$.
Note that even though $A_d$ and $B_d$ are supported in $supp(\Theta_d)$, 
$\sum_{|n|\leq \alpha_{2d+1}}u_n$ is not one on the support of $\Theta_d$, so we have to take a larger sum in the above expression.
We can see from the structure of the product that the trace norms satsify
a bound 
$$
\sup_{\Re(z) \in J, 0 < \Im(z) \leq 1} \|P_{n,0}(k,r)\|_1 \leq C_7(d, a, J), 
$$
since an inspection of the product in equation (\ref{eq:New810}), shows
that in any product, $z$ and $\{\omega_{\tilde{n}}, \tilde{n} \neq 0, n\}$  occurs at most to a power of $d+1$.
The uniform boundedness of the trace norm as a function of $z, ~\omega_{\tilde{n}}$ is clear since these variables are in compact sets. 
As for the finiteness of the trace norm itself, we note that any product
has $d+1$ factors from the set $\{A_j, B_j, j=0, \dots, d\}$, hence by the
claim in equation (\ref{eq:New805}), such a product is trace class. 

Using equations (\ref{eq:New705}, \ref{eq:New708}, \ref{eq:New801}, \ref{eq:New800}) and
 equation (\ref{eq:New811}) in equation (\ref{eq:New708}), we get,
using the fact that $P_{n,0}()$ are independent of $\omega_0, \omega_n$,
\begin{align}\label{eq:New812}
&T(L,\ell, Q, z)\nonumber \\
&=\displaystyle{\sum_{\substack{ \sum_{n\in \Lambda_{L+1}k_n} = \ell\\ k_n \geq 0}} \big(\substack{\ell \\ k_0, \dots, k_{\kappa(L+1)+1}}\big)} \int_{\RR^{\kappa(L+1)-1}} \sum_{|n|\leq \alpha_{2d+2}} \sum_{r_1, r_2 =0}^{d+1} tr\bigg(Q \bigg [ \nonumber \\
&\qquad\qquad \int u_0\big((H_{L+1}^\omega -z)^{-1} - (H_L^\omega -z)^{-1}\big) u_n \omega_n^{r_2} \omega_0^{r_1} \rho_n^{(k_n)}(\omega_n)\rho_0^{(k_0)}(\omega_0) d\omega_n d\omega_0 \nonumber \\ 
&\qquad\qquad\qquad \bigg] P_{n,0}(k,r,\omega)\bigg) \prod_{m\in \Lambda_{L+1}\setminus\{0,n\}} \rho_m^{(k_m)}(\omega_{m})d\omega_{m}.
\end{align}

We now estimate the absolute value of the trace in equation (\ref{eq:New812})
using the Theorem \ref{thm:remarkable}(1) for bounding the norm of the
integral with respect to $\omega_n, \omega_0$, since $2s <\tau$. 
\begin{align}\label{eq:New813}
&|T(L, \ell, Q, z)|  \nonumber \\ 
&\leq \displaystyle{\sum_{\substack{ \sum_{n\in \Lambda_{L+1}k_n} = \ell\\ k_n \geq 0}} \big(\substack{\ell \\ k_0, \dots, k_{\kappa(L+1)+1}}\big)} \int_{\RR^{\kappa(L+1)-2}}\sum_{|n|\leq \alpha_{2d+2}} \sum_{r_1, r_2 =0}^{d+1} \|Q\| \|P_{n,0}(k,r,\omega)\|_1 \nonumber \\ 
&  \bigg[ \int \big\|(u_0 + u_n)^\half \big((H_{L+1}^\omega -z)^{-1} - (H_L^\omega -z)^{-1}\big) {(u_n+u_0)}^\half\big\|^s \phi_R(\omega_0)\phi_R(\omega_n)  d\omega_n d\omega_0 \bigg]   \nonumber \\
&\qquad\qquad \prod_{m\in \Lambda_{L+1}\setminus\{0,n\}} |\rho_m^{(k_m)}(\omega_{m})|d\omega_{m}.
\end{align}
In the above inequality we also used the fact that $u_0 (u_0 + u_n)^{-\half},
u_n (u_0 + u_n)^{-\half}$ are both bounded uniformly in $n$ and replaced
$u_0, u_n$ by $(u_0 + u_n)^{\half}$ on either side of the resolvents.

We would prefer to work with probability measures in  above equation, so we normalize $|\rho_m^{(k_m)}(x)|dx$ by their $L^1$ norm. 
We also do the same for $\phi_R$. 
We then follow the steps involved in obtaining the inequality 
(\ref{eq:New343}). We set
$\eta(m, \rho) = (\sup_{n \in \ZZ^d, k_n \leq m} \|\rho_n^{k_n}\|_1 + \|\rho_n^{k_n}\|_\infty) + \|\phi_R\|_1$ to get, 
\begin{align}\label{eq:New814}
&|T(L, \ell, Q, z)|  \nonumber \\ 
&\leq \displaystyle{\sum_{\substack{ \sum_{n\in \Lambda_{L+1}k_n} = \ell\\ k_n \geq 0}} \big(\substack{\ell \\ k_0, \dots, k_{\kappa(L+1)+1}}\big)} \sum_{|n|\leq \alpha_{2d+2}}  C_9(a,d,J, \eta(\rho,m)) \nonumber \\  
& \qquad \cdot \mathbb{E}_{\mathbb{P}_{(k)}} \bigg[ \|(u_0 + u_n)^\half \big((H_{L+1}^\omega -z)^{-1} - (H_L^\omega -z)^{-1}\big) (u_n+u_0)^\half\|^s \bigg], 
\end{align}
where $\mathbb{E}_{\mathbb{P}_{(k)}}$ is the expectation with respect to the probability density 
$$\frac{\phi_R(\omega_0)d\omega_0}{\norm{\phi_R}_1}\frac{\phi_R(\omega_n)d\omega_n}{\norm{\phi_R}_1} \prod_{m\in \Lambda_{L+1}\setminus\{0,n\}} \frac{|\rho_m^{(k_m)}(\omega_{m})|}{\|\rho_m^{(k_m)}\|_1}d\omega_{m}.$$
We define a smooth radial function $0 \leq \Psi \leq 1$ such that 
$$
\Psi(x) = \begin{cases} 1, ~~ |x| \leq L/2, \\ 0, ~~ |x| > L/2 + 4 \end{cases}.
$$
Then $\Psi_L \sqrt{u_0 + u_n} = \sqrt{u_0 + u_n}, ~~ |n| \leq \alpha_{2d+2}$.
Following the steps similar to obtaining the inequality (\ref{eq:New806}),
using the relation $(H_{0,L} +a) R_{L,a}^0 = Id$, we have  
\begin{align}\label{eq:New815}
&(u_0 + u_n)^\half \big((H_{L+1}^\omega -z)^{-1} - (H_L^\omega -z)^{-1}\big) (u_n+u_0)^\half \nonumber \\
&= (u_0 + u_n)^\half \big((H_{L+1}^\omega -z)^{-1}[\Psi_L, H_0] (H_L^\omega -z)^{-1}\big) (u_n+u_0)^\half \nonumber \\
& = (u_0 + u_n)^\half (H_{L+1}^\omega -z)^{-1}[\Psi_L, H_0]\big[R^0_{L,a}+(H^\omega_L-z)^{-1}-R^0_{L,a} \big]  (u_n+u_0)^\half  \nonumber\\
& =  (u_0 + u_n)^\half (H_{L+1}^\omega -z)^{-1}[\Psi_L, H_0] R_{L,a}^0 \big( I + (z + a - V_L^\omega) (H_L^\omega -z)^{-1}\big) (u_n+u_0)^\half \nonumber\\
&= (u_0 + u_n)^\half (H_{L+1}^\omega -z)^{-1}
\bigg[ -\sum_{i=1}^d \frac{\partial^2}{\partial x_i^2} \Psi_L + 2\sum_{i=1}^d\bigg(\frac{\partial}{\partial x_i} \Psi_L\bigg)\bigg(-i\frac{\partial}{\partial x_i} + A_i\bigg)\bigg]  \nonumber \\
& \qquad R_{L,a}^0 \big( I + (z + a - V_L^\omega) (H_L^\omega -z)^{-1}\big) (u_n+u_0)^\half.
\end{align}
We take a smooth bounded radial function $0 \leq \Phi_L \leq 1$ 
which is 1 in a neighbourhood
of $L/2 \leq r \leq L/2 + 4$ and zero outside a neighbourhood of radial width
10. Then using the fact that 
\begin{align} \label{eq:New8140}
&\Phi_L \bigg(\sum_{i=1}^d \frac{\partial^2}{\partial x_i^2} \Psi_L\bigg) =
\bigg(\sum_{i=1}^d \frac{\partial^2}{\partial x_i^2} \Psi_L\bigg) \nonumber \\
&\Phi_L \bigg(\frac{\partial}{\partial x_i} \Psi_L\bigg) = \bigg(\frac{\partial}{\partial x_i} \Psi_L\bigg), ~ for ~~ all ~~ i=1,\dots, d
\end{align}  
and \eqref{eq:New815}, we can now bound the expectation in the inequality (\ref{eq:New814}),by
\begin{align}\label{eq:New816}
&\mathbb{E}_{\mathbb{P}_{(k)}} \bigg[ \|(u_0 + u_n)^\half \big((H_{L+1}^\omega -z)^{-1} - (H_L^\omega -z)^{-1}\big) (u_n+u_0)^\half\|^s \bigg] \nonumber \\
&\leq  \mathbb{E}_{\mathbb{P}_{(k)}} \bigg[ \|(u_n+u_0)^\half (H_{L+1}^\omega - z)^{-1} \Phi_L\|^s\nonumber \\
&\qquad \bigg\|\bigg[ \bigg(-\sum_{i=1}^d \frac{\partial^2}{\partial x_i^2} \Psi_L\bigg) + 2\sum_{i=1}^d\bigg(\frac{\partial}{\partial x_i} \Psi_L\bigg)\bigg(-i\frac{\partial}{\partial x_i} + A_i\bigg)\bigg] R_{L,a}^0\bigg \|^s \nonumber \\
&\qquad \big(1 + |z|+a + \|V_L^\omega\|_\infty) \|\chi_{\Lambda_L} (H_L^\omega - z)^{-1} \sqrt{u_0 + u_n} \|^s \bigg].
\end{align}
Then using Cauchy-Schwartz inequality and Hypothesis \ref{hyp:expLocHypCts} 
we  get an exponential bound for the first factor, a uniform bound for the
second factor after noting that $dist(supp(\Phi_L, \{n : |n| \leq \alpha_{2d}+1\}) \geq L/4$, $\|\Lambda_L\| \leq (2L)^d $,
we get the estimate
\begin{align}\label{eq:New817}
&\sup_{z: \Re(z) \in J, \Im(z) \leq 1} \Ex_\omega \bigg[ \|(u_0 + u_n)^\half \big((H_{L+1}^\omega -z)^{-1} - (H_L^\omega -z)^{-1}\big) (u_n+u_0)^\half\|^s \bigg] \nonumber \\
&\qquad \leq C_{10}(a, J, d) L^d e^{-\xi_{2s} L}. 
\end{align}
Using this inequality in (\ref{eq:New814}) we get the bound
\begin{align}\label{eq:New818}
&\sup_{\substack{z: \Re(z) \in J, \Im(z) \leq 1,\\  Q \\ \ell \leq m}} |T(L, \ell, Q, z)| \nonumber \\
&\leq C_{11}(a, d, J, \eta(\rho, m)) (L+1)^{d(m+1} e^{-\xi_{2s} L}, 
\end{align}
as the combinatorial sum
\begin{align*}
\displaystyle{\sum_{\substack{\sum_{n\in\Lambda_{L+1}} k_n = \ell\\ k_n \geq 0}} \big(\substack{\ell \\ k_0, \dots, k_{\kappa(L+1)}}\big)}
\end{align*}
is easily seen to add up to $(L+1)^{d\ell}$, which is still polynomial in $L$.
This bounds shows the summability in equation (\ref{eq:New704}) completing the proof. 

\end{proof}

{\noindent \bf Acknowledgement:}  We thank the referee of this paper for very detailed, numerous and critical comments that helped us improve the presentation of the 
results.
\appendix
\section{Appendix}

We collect a few Lemmas in this appendix that are used in the main part
of the paper.  All these Theorems are well known and proved elsewhere
in the literature, but we state them in the form we need and also give their
proofs for the convenience of the reader.

\begin{lemma}\label{derivatives}
Consider a positive function $\rho \in L^1(\RR, dx)$ and $J \subset \RR$ 
an interval. Let
$F(z) = \int \frac{1}{x - z} ~ \rho(x) dx $.  Then, for any $m \in \NN$,  
$$
\supess_{x \in J} \left|\frac{d^m}{dx^m}\rho \right|(x) < \infty
$$
whenever
$$
\sup_{z \in \CC^+, ~ \Re (z) \in J} \left|\frac{d^{m}}{dx^{m}} \Im( F )\right|(z) < \infty.
$$
\end{lemma}
\begin{proof}
 Since $\rho(x) dx$ is a finite positive measure, $F$ is analytic 
in $\CC^+$, and the assumption on $F$ implies that functions 
$\frac{d^\ell}{dz^\ell}\Im(F)$ are bounded Harmonic function in the strip
$\{z \in \CC^+ : \Re(z) \in J\}$, $0 \leq \ell \leq m$.  
Therefore the boundary values  
$$
h_\ell(E) = \lim_{\epsilon \rightarrow 0} \frac{d^\ell}{dz^\ell} \Im(F)(E+i\epsilon) 
$$
exist Lebesgue almost every $E \in J$ and $h_\ell$ are essentially bounded in 
$J$, $0 \leq \ell \leq m$.  For any $E_0 \in J$ for which $h_\ell(E_0)$
is defined for all $0 \leq \ell \leq m$, we have  
$0 \leq \ell \leq m-1$, 
\begin{equation}\label{eq:New306}
\frac{\partial^\ell}{\partial x^\ell}(\Im F)(E+i\epsilon) - \frac{\partial^\ell}{\partial x^\ell}\Im(F)(E_0+i\epsilon) = \int_{E_0}^E \frac{\partial^{\ell+1}}{\partial x^{\ell +1}}(\Im F)(x+i\epsilon) ~ dx, ~~ E \in J. 
\end{equation}
Since the integrands above are Harmonic functions in the strip, their boundary
values exist, they are uniformly bounded in the strip, so by the dominated
convergence theorem the integral converges to 
$$
\int_{E_0}^E h_{\ell +1}(x) ~ dx, ~~ E \in J.
$$
On the other hand the left hand side of equation (\ref{eq:New306}) 
converges to $h_{\ell}(E) - h_\ell(E_0)$,
showing that $h_\ell(E)$ is differentiable in $J$.  Since, 
$\rho(x) = \frac{1}{\pi}h_0(x), ~ x \in J$, a simple induction argument
now gives the Lemma. 

\end{proof}
\begin{lemma}\label{lem:semiGrpNrmbdd}
On a separable Hilbert space $\Hi$, let $A$ and $B$ be two bounded operators generating strongly differentiable contraction semi-groups $e^{tA}, e^{tB}$ respectively, then for any $0 < s < 1$,
$$\norm{e^{ t A}-e^{ t B}}\leq 2^{1-s}|t|^s \norm{A-B}^s.$$
\end{lemma}
\begin{proof}
Since $e^{tA}, e^{tB}$ are strongly differentiable,  the fundamental
Theorem of calculus gives the bound, 
\begin{align*}
\norm{e^{ t A}-e^{ t B}}=\norm{ \int_0^t e^{(t-s)A} (A - B)e^{sB} ~ ds }
\leq |t| \norm{A - B}.
\end{align*}
Since $e^{tA}, e^{tB}$ are contractions we have the trivial bound
$$
\norm{e^{tA} - e^{tB}} \leq 2,
$$
so the Lemma follows by interpolation.

\end{proof}
\begin{lemma}\label{lem:interchangingIntegrals}
Let $g$ be a probability density with a $\tau$-H\"older continuous derivative.
Suppose $A$ is a bounded operator on a separable Hilbert space $\Hi$ with 
$\Im(A) > 0$ and satisfies
$$\norm{(A+\lambda I)^{-1}}<C<\infty, ~~\lambda\in supp(g). $$
Then
\begin{equation}\label{eq:New309}
\int g(\lambda)(A+\lambda I)^{-1}d\lambda=- \int_0^\infty e^{i t A}\left(\int g(\lambda)e^{i t \lambda}d\lambda\right)dt.
\end{equation}
\end{lemma}
\begin{proof}
Since $(A+\lambda I)^{-1}$ is bounded we have, in the strong sense,
$$(A+\lambda I)^{-1}=\lim_{\epsilon\downarrow 0}(A+\epsilon+\lambda I)^{-1}.$$
Since $\Im(A) >0$,  the bounded operator $(A+i\epsilon)$ is the 
generator of a contraction semi-group, so using 
\cite[Corollary 1, Section IX.4]{MR617913} we have 
\begin{align}\label{eq:New308}
\int g(\lambda) (A+i\epsilon+\lambda I)^{-1} d\lambda &=\int g(\lambda)\int e^{i t (A+i\epsilon+\lambda I)} dt d\lambda \nonumber \\
&= \int \int g(\lambda) e^{(-\epsilon + \lambda)t} e^{i t A} dt d\lambda.
\end{align}
Since $g$ has a $\tau$-H\"older continuous derivative, its Fourier transform
is a bounded integrable function.  Therefore by Fubini we can interchange
the $\lambda$ and $t$ integrals on the right hand side of the above equation
to get the right hand side of equation (\ref{eq:New309}).
On the other hand using the fact that $\norm{(A+\epsilon+\lambda I)^{-1}}<2C$ for $0<\epsilon<\frac{1}{2C}$ and $g$ is a probability density, we have
\begin{align*}
&\lim_{\epsilon\downarrow 0} \int g(\lambda) (A+i\epsilon+\lambda I)^{-1} d\lambda
\qquad=\int g(\lambda) \left[ \lim_{\epsilon\downarrow 0}(A+i\epsilon+\lambda I)^{-1}\right] d\lambda\\
&\qquad=\int g(\lambda) (A+\lambda I)^{-1} d\lambda.
\end{align*}
This set of equalities when applied to the left hand side of the equation
(\ref{eq:New309}) gives the Lemma after letting $\epsilon$ go to zero. 
\end{proof}
We give the Lemma below which is a consequence of proofs of results in
Stollmann \cite{MR2654125} and Combes-Hislop-Klopp \cite{MR2042528}.  These papers essentially
prove the result, but we write it here since it does not occur in the form
we need to use.  
\begin{lemma}\label{specavg}
Suppose $A$ is a self-adjoint operator on a separable Hilbert space $\Hi$
and suppose $B$ is a non-negative bounded operator. Consider the operators
$A(t) = A + t B, ~~ t \in \RR$, $ \phi \in Range(B)$ and $\nu^\phi_{A(t)}$
the spectral measure of $A(t)$ associated with the vector $\phi$. Suppose
$\mu$ is a finite absolutely continuous measure with bounded density, then
\begin{equation}
\sup_{z \in \CC^+}  \int \Im ( \langle \phi, (A(t) - z)^{-1}\phi \rangle ) d\mu(t) < \infty.
\end{equation}   
In particular the measure $\int \nu^\phi_{A(t)} ~ dt$ has bounded density.
\end{lemma}
\begin{proof}
We set $\tilde{\nu} = \int \nu^\phi_{A(t)} ~ dt$, then $\tilde{\nu}$
is a positive finite measure.  We recall that the modulus of continuity of a 
measure $\nu$ is defined as
$$
s(\nu, \epsilon) = \sup \{ \nu([a, a+\epsilon]) : a \in \RR\}.
$$ 
This definition immediately implies that an absolutely continuous measure
$\mu$ with bounded density $\rho$, satisfies $s(\mu, \epsilon) \leq \|\rho\|_\infty  \epsilon$.
Therefore the Theorem 3.3 of Stollman \cite{MR2654125}, implies 
that 
$$
s(\tilde{\nu}, \epsilon) \leq 6\|B\|\|\phi\| s(\mu, \epsilon) \leq C\|\rho\|_\infty  \epsilon.
$$
This inequality implies that the density of 
$\tilde{\nu}$ is bounded.  Since the function 
$$
F(z) = \int \Im ( \langle \phi, (A(t) - z)^{-1}\phi \rangle ) ~ d\mu(t)
 = \int \Im (\frac{1}{x - z}) ~ d\tilde{\nu}
$$
is positive Harmonic in $\CC^+$, by the maximum principle its supremum is
attained on $\RR$. The boundary values of $F$ on $\RR$ exist 
and equal the density of the measure
$\tilde{\nu} = \int \nu^\phi_{A(t)} ~ dt$ Lebesgue almost everywhere
, by Theorem 1.4.16 of Demuth-Krishna \cite{MR2159561},  giving the result.
\end{proof}

\begin{lemma}\label{lem:traceclass}
Consider the operators $H^\omega$, $H_\Lambda^\omega$ given in 
equation (\ref{eq:ranOpCont}) and the following discussion.
Then for any finite $E \in \RR$, the operators
$u_0 E_{H^\omega_\Lambda}((-\infty, E))$,  $u_0 E_{H^\omega}((-\infty, E))$ 
are trace class for all $\omega$.  The traces of these operators
are uniformly bounded in $\omega$ for fixed $E$.
\end{lemma}
\begin{proof}
We will give the proof for $H^\omega$, the proof for the others is similar.
The hypotheses on  $H^\omega$ imply that it is  bounded below and 
the pair $H_0, H^\omega$ are relatively
bounded with respect to each other, being bounded perturbations of each
other, the operators $(H_0+a)^d E_{H^\omega}((-\infty, E))$ are bounded
for any fixed $(E, a, \omega)$.  So taking $a$ in the resolvent
set of $H_0$ and using the fact that $u_0(H_0+a)^{-d} $ is trace class 
we see that 
$$u_0 E_{H^\omega}(-\infty, E)=u_0 (H_0+a)^{-d}(H_0+a)^{d} E_{H^\omega}(-\infty, E),$$
is a  product of a trace class operator and a bounded operators 
for each fixed $(\omega, a, E)$ with $a$ positive and large.  Therefore 
$u_0 E_{H^\omega}(-\infty, E)$ is  also trace class for each $E, \omega$.
The uniform boundedness statement is obvious from the
assumptions on the random potential.
\end{proof}

\bibliographystyle{plain}

\end{document}